\documentclass[12pt,leqno]{amsart}
\usepackage{times}
\usepackage{amsmath,amssymb,amscd,amsthm}
\usepackage{latexsym}
\usepackage{epsfig}

\newtheorem{theorem}{Theorem}
\newtheorem{lemma}{Lemma}
\newtheorem{proposition}{Proposition}

\newtheorem{corollary}{Corollary}

\newtheorem{claim}{Claim}

\newcommand{\qq}{\quad\quad}

\newcommand{\nab}{\langle\nabla\rangle}

\newcommand{\naba}{\nab^\al}
\newcommand{\nabam}{\nab^{-\al}}

\newcommand{\f}[2]{\frac{#1}{#2}}

\newcommand{\al}{\alpha}
\newcommand{\be}{\beta}
\newcommand{\ga}{\gamma}

\newcommand{\de}{\delta}

\newcommand{\ve}{\varepsilon}

\newcommand{\vp}{\varphi}

\newcommand{\mw}{\mathbf w}

\newcommand{\rn}{{\mathbf R}^n}

\newcommand{\rone}{\mathbf R^1}
\newcommand{\rtwo}{\mathbf R^2}
\newcommand{\rthree}{\mathbf R^3}

\newcommand{\zz}{\mathbf Z}

\newcommand{\cn}{{\mathcal N}}

\newcommand{\cf}{\mathcal F}

\newcommand{\ch}{\mathcal H}
\newcommand{\cx}{\mathcal X}

\newcommand{\p}{\partial}

\newcommand{\med}{\textnormal{med}}

\newcommand{\beq}{\begin{equation}}
\newcommand{\eeq}{\end{equation}}
\newcommand{\beqna}{\begin{eqnarray*}}
\newcommand{\eeqna}{\end{eqnarray*}}
\newcommand{\beqn}{\begin{equation*}}
\newcommand{\eeqn}{\end{equation*}}
\newcommand{\bp}{\begin{proof}}
\newcommand{\ep}{\end{proof}}
\newcommand{\bprop}{\begin{proposition}}
\newcommand{\eprop}{\end{proposition}}
\newcommand{\bt}{\begin{theorem}}
\newcommand{\et}{\end{theorem}}
\newcommand{\bex}{\begin{Example}}
\newcommand{\eex}{\end{Example}}
\newcommand{\bc}{\begin{corollary}}
\newcommand{\ec}{\end{corollary}}
\newcommand{\bcl}{\begin{claim}}
\newcommand{\ecl}{\end{claim}}
\newcommand{\bl}{\begin{lemma}}
\newcommand{\el}{\end{lemma}}

\begin{document}

\title
[l.w.p. for quadratic Schr\"odinger equations in ${\mathbf R^{1+1}}$]
{Local well-posedness for quadratic Schr\"odinger equations in $\mathbf{R^{1+1}}$: a normal form approach}

\author{Seungly Oh, Atanas Stefanov}

\address{Seungly Oh, 
405, Snow Hall, 1460 Jayhawk Blvd. , 
Department of Mathematics,
University of Kansas,
Lawrence, KS~66045, USA}
\address{Atanas Stefanov, 405, Snow Hall, 1460 Jayhawk Blvd. , 
Department of Mathematics,
University of Kansas,
Lawrence, KS~66045, USA}
\date{\today}

\thanks{Oh and Stefanov are partially supported by  NSF-DMS 0908802 }

\subjclass[2000]{35Q55 (35A07, 35B30, 35B45,  35R05)}

\keywords{Schr\"odinger equations, normal forms}

\begin{abstract}
For the Schr\"odinger equation $u_t+i u_{xx}=\nab^\be[u^2]$, $\be\in (0,1/2)$, we establish local well-posedness in $H^{\be-1+}$ (note that if $\be=0$, 
this matches, up to an endpoint,  the sharp result of Bejenaru-Tao, \cite{BT}). Our approach differs significantly from the previous ones in that we use normal form transformation to analyze the worst interacting terms in the nonlinearity and then show that the remaining terms are (much) smoother. In particular, 
this allows us to conclude that $u-e^{-i t \p_x^2} u(0)\in H^{-\f{1}{2}}(\rone)$, even though $u(0)\in H^{\be-1+}$. 

 In addition and as a byproduct of our normal form analysis, we obtain a Lipschitz continuity property in $H^{-\f{1}{2}}$ 
of the solution operator (which originally acts on $H^{\be-1+}$), which is new even in the case $\be=0$. As an easy 
corollary, we obtain local well-posedness results for $u_t+ i u_{xx} =   [\nab^{\beta} u]^2$.  

Finally, we sketch an approach to obtain similar statements for the equations $u_t+i u_{xx}=\nab^\be[u\bar{u}]$ and $u_t+i u_{xx}=\nab^\be[\bar{u}^2]$. 
\end{abstract}

\maketitle
\date{today}

\section{Introduction}
In this paper, we will be concerned with local solutions of the quadratic  Schr\"odinger equations  
\begin{equation}
\label{s:1}
\left| 
\begin{array}{l}
u_t+ i u_{xx}= Q(u, u): \qq (t,x)\in \rone_+\times \rone\\
u(0,x)=u_0  
\end{array}\right.
\end{equation}
The problem has received a lot of attention in the last twenty years and a full account of the appropriate results  and open questions is beyond of the scope of 
the current project. We will however outline a selected list of recent works, which has some  bearing on the problem that we are   studying. 

The classical results of the subject go back to Tsutsumi, \cite{Tsu}, which establishes local well-posedness for data in $H^s, s\geq 0$ for all quadratic nonlinearities (i.e. 
$|Q(u,u)|\leq C|u|^2$). This is in a way optimal, since for Hamiltonian models (i.e. with $Q(e^{i \theta} u, e^{i \theta} u)=e^{i \theta} Q(u,u)$), it is well-known that there is ill-posedness in $H^s, s<0$ - this is in the work of Kenig-Ponce-Vega, \cite{CPV2}, see also Christ, \cite{christ} and Christ-Colliander-Tao, \cite{CCT}  
for further results in this direction. 

For the non-Hamiltonian model, several different models have been considered in the literature, the most popular being  $Q(u,u)=u^2, u\bar{u}, \bar{u}^2$. Each of these comes with its own specifics and the corresponding local well-posedness results reflect that.  Regarding the cases $u^2, \bar{u}^2$, it has been show by Kenig-Ponce-Vega \cite{CPV2}, that these are well posed in $H^{-\f{3}{4}+}$ by means of bilinear estimates in $X^{s,b}$ spaces. Moreover, such bilinear estimates necessarily fail at the critical index $-3/4$, \cite{CPV2,NTT}. Regarding the nonlinearity $u\bar{u}$, it has been shown, that the problem is well-posed in $H^{-\f{1}{4}+}$, \cite{CPV2}  and this turns out to be sharp\footnote{At least as far as the uniform continuity of the solution map goes} ,\cite{LW}. On the other hand, the results for $u\bar{u}$ may be pushed down to the really sharp index $s=-1/2+$, if one is willing to put some homogeneous Sobolev space requirements on the low-frequency portion of the data, \cite{LW}

Regarding the nonlinearity $u^2$, the results of \cite{CPV2} were 
extended to the sharp index $s\geq -1$,by Bejenaru-Tao  \cite{BT}, see also the work  Bejenaru-Da Silva, \cite{BS} for the same result in two spatial dimensions. As we have mentioned, the spaces $X^{s,b}$ by themselves, could not accommodate such low regularity of solutions, so the authors had to come up with  further refinements of these spaces, in which they were able to preform their fixed point arguments, see also Nishimoto, \cite{Nishimoto2} for interesting commentary on these developments. For the nonlinearity $\bar{u}^2$, it has been shown that similar techniques may be used to obtain l.w.p. in $h^s, s\geq -1$, Nishimoto, \cite{Nishimoto1}. 
It also should be noted that in all of these papers (with the exception of  \cite{Nishimoto2}), it is 
hard to show optimal l.w.p. for Schr\"odinger equations with 
nonlinearity of the form $Q(u,u)=c_1 u^2 +c_2 \bar{u}^2$, due to the specifics of the approaches. The result in \cite{Nishimoto2} achieves this goal, at the expense of further layer of complexity, involved in the definition of the spaces and the corresponding bilinear estimates that need to be shown.   

Our main result concerns the following specific 
 generalization of the quadratic Schr\"odinger equation \eqref{s:1} 
\begin{equation}
\label{1}
\left| 
\begin{array}{l}
u_t+ i u_{xx} = \nab^{\beta} [u^2]: \qq (t,x)\in \rone_+\times \rone\\
u(0,x)=u_0 \in H^{-\alpha},
\end{array}\right.
\end{equation}
where $\be\geq 0$. 
This model has been well-studied over the years, mainly the case $\be=1$. We should first mention, that the corresponding equation is ill-posed, in the sense that the flow map $u_0\to u$ experiences norm inflation, Christ \cite{christ}, 
see also \cite{CCT} for related results.In the work of Stefanov, \cite{stefanov}  
existence of weak solutions in $H^1$ was shown, under the additional smallness requirement $\sup_x |\int_{-\infty}^x u_0(y) dy| <<1$. Similar results\footnote{again for data small in $L^1$ sense and so that it belongs to some smooth modulation spaces}  (in $\rn, n\geq 1$),  were obtained for the more general Ginzburg-Landau equation in the work of Han-Wang-Guo, \cite{HWG}. Finally, we mention some recent local well-posedness results, which were obtained for (not necessarily small) data in 
 weighted Sobolev spaces by Bejenaru, \cite{Bej1, Bej2} 
 and Bejenaru-Tataru, \cite{BTat}. 
 
 While some of the   positive results mentioned above surely will extend to the case $\be\in (0,1)$, it seems that this model has not 
 been considered in the literature. One of the purposes of the current paper is to address the question for local well-posedness of this problem. An alternative goal is to develop an alternative proof of the known results in the case $\be=0$, which is within the framework of the standard $X^{s,b}$ spaces. We achieve that by the technique of normal forms. This method has been used extensively in the last twenty years, to treat global small solutions of  models with low-power nonlinearities -  see the pioneering work of Shatah, \cite{shatah} and more recent developments in  Germain-Masmoudi-Shatah \cite{GMS}, \cite{GMS1} and  Shatah, \cite{shatah1}.

 Informally, the method starts with a preconditioning of the equation first (via a  change of variables - normal form). The particular type of the normal form is a bilinear pseudo-differential operator, which solves explicitly the Schr\"odinger equation with right-hand side, which consists of 
 the most troublesome terms in the nonlinearity, see \eqref{30} below. Then, one needs to argue that the remaining terms of the solution (which are obtained through a fixed point argument, involving themselves, the normal form and various interactive terms) are better behaved - in our case, the are (much) smoother than the free solution. As a byproduct of this approach, we obtain more precise information on these  correction terms (see \eqref{s:5}). {\it To the best of our knowledge, this paper is the first one that treats low regularity solutions, with the method of normal forms.}
 
 Our main results recover (up to an endpoint) the sharp results of \cite{BT}, \cite{Nishimoto2} in the case $\be=0$, but of course covers also  the new 
 cases $\be \in (0,1/2)$, where the results are also arguably sharp.  We also obtain the Lipschitz property \eqref{a:200} of the solution map, which is a new feature, even in the case $\be=0$. \\
More specifically, 
\begin{theorem}
\label{theo:1}
Let $\be\in [0,\f{1}{2})$ and $\al: \f{1}{2}< \al <1-\be$. Then the equation \eqref{1} is locally well-posed in 
$H^{-\al}(\rone)$.  More specifically, for every $u_0\in H^{-\al}(\rone)$, there exists a non-trivial $T>0$, so that the equation \eqref{1} has an unique solution 
$u\in C([0,T), H^{-\al})$. 

Moreover, for fixed $\de: 0<\de< \f{1-\al-\be}{10}$, there is the following decomposition 
\begin{equation}
\label{s:5}
u=e^{-i t \p_x^2} u_0+ h+ w,
\end{equation} 
where $h\in L^\infty_t H^{\f{1}{2}-\al}_x\cap X^{1-\al-\de, \de}$, $w\in 
X^{-\f{1}{2} , \f{1}{2}+\de}$. In particular $u-e^{-i t \p_x^2} u_0 
\in L^\infty_t H^{-\f{1}{2}}_x$. \\ \\ 
\underline{We also have the following Lipschitz property of the solution map of \eqref{1}}.  
Let $N>0$ and $u_0, v_0 \in H^{-\al}(\rone): \|u_0\|_{H^{-\al}}<N, 
\|v_0\|_{H^{-\al}}<N $, so that $u_0-v_0 \in H^{-\f{1}{2}}$. Then the corresponding solutions (defined on a common non-trivial time interval $(0,T_N)$) satisfy 
\begin{equation}
\label{a:200}
\|u-v\|_{L^\infty_T H^{-\f{1}{2}}_x}\leq C_N \|u_0-v_0\|_{H^{-\f{1}{2}}_x}.
\end{equation}
where $C_N$ depends only on $N$. 
\end{theorem}
{\bf Remarks:} 
\begin{itemize}
\item There are of course well-posedness results  in the 
cases $\al\in (0,1/2)$ and they are easier to obtain. We chose to include only those with $\al>1/2$ in order to simplify our exposition.  
\item The Lipschitz property \eqref{a:200} is new even in the case $\be=0$.

\item We do not obtain l.w.p. for the case $\al=1-\be$, which in the case $\be=0$, will correspond to the endpoint case of $s=1$, considered  in \cite{BT}. 
Our arguments would imply such a statement, at least in the case of a  Besov-1 space version 
of the main result.

\item While our arguments fail at $\be\geq 1/2$, we cannot claim  sharpness in this regard. However, we very strongly suspect that this is the case. That is, we conjecture that some form of ill-posedness must occur, when one considers solutions to \eqref{1} with  $\be=1/2$. 

\end{itemize}
We also have the following corollary. Consider 
\begin{equation}
\label{s:10}
\left| 
\begin{array}{l}
z_t+ i z_{xx} =   \nab^{\beta} z  \nab^{\beta} z \qq (t,x)\in \rone_+\times \rone\\
z(0,x)=z_0  
\end{array}\right.
\end{equation}
Setting $u=\nab^{\beta} z$ yields the equation 
\begin{equation}
\label{s:20}
u_t+ i u_{xx} =   \nab^{\beta} [u^2],
\end{equation}
for $u$. By Theorem~\ref{theo:1}, we conclude that \eqref{s:20} is well-posed in $H^{-\al}$, for all $\f{1}{2}<\al<1-\be$. Therefore, in terms of $z$, we have well-posedness in $H^{-\al+\be}$ 
\begin{corollary}
\label{corol:1}
Let $\be\in [0,1/2)$ and $0<\ga<1-2\be$. The equation \eqref{s:10} is well posed in $H^{-\ga}$. 
\end{corollary}
The paper is organized as follows. In Section~\ref{sec:s2}, we introduce the $X^{s,b}$ spaces, as well as Tao's theory of bilinear estimates with $X^{s,b}$ entries, \cite{Tao1}. In Section~\ref{sec:s3}, we construct the normal forms and set the function spaces that will be used. In Section~\ref{sec:s32}, we provide the basic estimates for the normal form transformation. In Section~\ref{sec:s33}, we state and prove the bilinear and trilinear estimates, needed for the fixed point argument. This is where the main technical difficulties are present. In Section~\ref{sec:s34}, we conclude the proof, by reducing it to the bilinear and trilinear estimates proved in Section~\ref{sec:s2} and Section \ref{sec:s33}. In Section~\ref{last}, we give some ideas on how to approach the question for local well-posedness for the problem with nonlinearities of the form $\nab^{\beta}[u \bar{u}]$ and $\nab^{\beta}[\bar{u}^2]$. 
\section{Some notations and preliminaries}
\label{sec:s2}
\subsection{Littlewood-Paley projections and Paraproducts} \label{lp}
 
Introduce the Fourier transform and its inverse via 
\begin{eqnarray*}
& & \hat{h}(\xi)= \int_{\rone} h(x) e^{-i x\xi} dx, \\
& & h(x)=\f{1}{2\pi} \int_{\rone} \hat{h}(\xi) e^{i x\xi} dx.
\end{eqnarray*}

 Let  $\Phi: \rone \to \rone$ be a positive, smooth even function  supported in $\{ \xi: |\xi| \leq 2\}$, and $\Phi(\xi) =1$ for all $|\xi|\leq 1$.  Define $\varphi (\xi) =\Phi(\xi) - \Phi(2\xi)$, which is supported in the annulus $1/2 \leq |\xi| \leq 2$.  Clearly $\Phi(\xi) +\sum_{k\in \zz^+} \varphi(2^{-k} \xi) =1$ for all $\xi$.\\

The $k^{th}$ Littlewood-Paley projection is $\widehat{P_k f}(\xi) = \varphi (2^{-k} \xi) \widehat{f}(\xi)$ for $k\in \zz^+$.  Similarly, $\widehat{P_{\leq 0} f} (\xi) = \Phi(\xi) \widehat{f}(\xi)$, and for any subset $A \subseteq Z$, we denote $\widehat{P_{A} f}(\xi) = \sum_{k\in A} \varphi(2^{-k} \xi) \widehat{f}(\xi)$ with the summation replaced by $\Phi(\xi)$ when necessary.  The kernels of $P_k$, $P_{\leq 0}$ are uniformly integrable and thus $P_k, P_{\leq 0}: L^p \to L^p$ for $1\leq p\leq \infty$ with the bound equivalent to $||\widehat{\Phi}(\xi)||_{L^1}$ (independent of $k$).  We will often use the notation $g_k$ in place of $P_k g$.\\

Define the operator $\langle \nabla \rangle^{\alpha}: H^{s} \to H^{s-\alpha}$ by $\widehat{\langle \nabla \rangle^{\alpha} g} (\xi) = (1+\xi^2)^{1/2} \widehat{g}(\xi)$.  Then $\|g\|_{H^s} = \|\langle \nabla \rangle^s g\|_{L^2}$, so $\|g_k\|_{H^s} \sim 2^{ks} \|g\|_{L^2}$ and 
$\|g_{\leq 0}\|_{H^s}\sim \|g\|_{L^2}$. 

Note that in order to simplify the exposition,  we will use the notation somewhat loosely in the sense that $g_k$ will  always denote a function so that $\widehat{g_k}(\xi)=\psi(2^{-k}\xi) \hat{g}(\xi)$,  where $\psi\in C_0^\infty([1/2,2])$ (but  the function $\psi$ may change from line to line). 

Next, we introduce the following basic decomposition from the theory of the paraproducts.  For any two Schwartz functions $f,g$ and $k\in \zz$, 
\[
P_k (fg) = P_k\left(\sum_{l,m} f_l g_m \right)= P_k \left( \sum_{|l-m|\leq 3} f_l g_m\right) + 
P_k \left( \sum_{|l-m|>3} f_l g_m\right).
\]

Furthermore, in the first sum, we have the restriction $\min (l,m) > k-5$; and in the second sum, we have $|\max (l,m)-k| \leq 3$.  Otherwise the $\textnormal{supp } \widehat{f_l g_m}$ will be away from $\{ \xi: |\xi|\sim 2^k\}$, and thus $P_k (f_l g_m) =0$.  We refer to the first summand as "high-high interaction" terms, and the second summand as "high-low interaction" terms.

\subsection{$X^{s,b}$ spaces and bilinear $L^2$ multipliers.}

For $s,b \in \rone$ and a function~$h$ of one variable, we define the (inhomogeneous) $X^{s,b}_{\tau= h(\xi)}$ spaces  to be the completion of $\mathcal{S}(\mathbf{R \times R})$ with respect to the norm\footnote{All of this makes perfect sense in higher spatial dimensions, but in this paper, we will confine our attention to the case $x\in\rone$}
\begin{equation}
\label{xsb1}
\|u(t,x)\|_{X^{s,b}_{\tau=h(\xi)}} = \left(\int_{\rone\times\rone} (1+\xi^2)^{s} (1+ |\tau- h(\xi)|)^{2b} |\widetilde{u}(\tau,\xi)|^2 \, d\tau \, d\xi\right)^{\frac{1}{2}}
\end{equation}
where $\displaystyle \widetilde{u} (\tau, \xi) = \int_{\rone \times \rone} u(t,x) e^{-i (t\tau + x\xi)}\, dt\, dx$.  Note that $X^{s,0}_{\tau=h(\xi)} = L^2_t H^s_x$.  \\

For Shr\"{o}dinger's equations, we use $h(\xi) = \pm \xi^2$.  For convenience, we will shorten the notation $X^{s,b} := X^{s,b}_{\tau=\xi^2}$.  Sometimes it will be necessary to use the space~$X^{s,b}_{\tau=-\xi^2}$, and it will be distinguished properly.\\

The usefulness of $X^{s,b}$ spaces come from the fact that it measures smoothness of a function not only in the classical Sobolev sense, but also in terms of the interaction between space and time frequencies.  More specifically, for $f\in L^2$, the free solution~$e^{-it\p_x^2} f \in L^2$ has no additional smoothness in the Sobolev scale than the initial data, but it is indeed very smooth in the weight~$\tau-\xi^2$ (and in fact lives, on the Fourier side, exactly on the parabola $\tau=\xi^2$).  \\

$X^{s,b}$ spaces can be used to study local-in-time solutions.  This is usually done by multiplying the solution by a smooth cutoff function.  
There is  the following classical estimate
\begin{proposition}\label{xsb}
Let $\Phi$ be a smooth cutoff adapted to $(-2,2)$ as defined as in Section~\ref{lp}.  
If $u(t,x)$ solves $(\p_t + i\p_x^2) u = F$ and 
$u(0,x) = u_0(x)$, then for $T>0$, $s\in \rone$ and $\delta>0$, there exists $C_{T, \de}$, so that 
\[
\|\Phi (t/T) u\|_{X^{s,\frac{1}{2}+\delta}} \leq C_{T,\delta} (||u_0||_{H^s} + ||F||_{X^{s,-\frac{1}{2}+\delta}}).
\]
\end{proposition}

We give some additional properties of $X^{s,b}$ norms.\\

Using $\widetilde{\overline{u}}(\tau,\xi) = \overline{\widetilde{u} (-\tau,-\xi)}$, we obtain
\begin{equation}\label{bar}
\|\overline{u}(t,x)\|_{X_{\tau=\xi^2}^{s,b}} = \int_{\rone\times \rone} (1+\xi^2)^{s} (1+ |\tau + \xi^2|)^{2b} |\widetilde{u}(\tau,\xi)|^2 \, d\tau \, d\xi =\|u\|_{X_{\tau=-\xi^2}^{s,b}}.
\end{equation}

Also we have the duality relation: $(X_{\tau=\xi^2}^{s,b})^* = X^{-s,-b}_{\tau=-\xi^2}$, which also allows us to define the equivalent norm
\begin{equation}\label{dual1}
\|u(t,x)\|_{X^{s,b}_{\tau=\xi^2}} = \sup_{\|v\|_{X_{\tau=-\xi^2}^{-s,-b}}=1} \left| \int_{\rone \times\rone} u(t,x) v(t,x)\, dt\,dx \right|.
\end{equation}

Note that using \eqref{bar}, we can also write
\begin{equation} \label{dual2}
\|u(t,x)\|_{X^{s,b}_{\tau=\xi^2}} = \sup_{\|v\|_{X_{\tau=\xi^2}^{-s,-b}}=1} \left| \int_{\rone \times\rone} u(t,x) \overline{v}(t,x)\, dt\,dx \right|.
\end{equation}

Observe that for $u,v\in X^{s,b}_{\tau=\xi^2}$, 
\begin{align}
\|u v\|_{X^{s,b}_{\tau=\xi^2}} &= \sup_{\|w\|_{X_{\tau=-\xi^2}^{-s,-b}}=1} \left| \int_{\rone \times \rone} u(t,x) v(t,x) w(t,x)\, dt\,dx \right| \notag\\
	&= \sup_{\|v\|_{X_{\tau=-\xi^2}^{-s,-b}}=1} \left| \int_{\rone \times \rone} \widehat{u v}(\tau,\xi) \widehat{w}(-\tau,-\xi)\, dt\,dx \right| \notag\\
&= \sup_{\|w\|_{X_{\tau=-\xi^2}^{-s,-b}}=1} \left| \int_{\begin{array}{l}\tau_1+\tau_2+\tau_3=0;\\ \xi_1 + \xi_2+\xi_3=0\end{array}} \widehat{u} (\tau_1,\xi_1) \widehat{v}(\tau_2,\xi_2) \widehat{w}(\tau_3,\xi_3)\, d\sigma \right|. \label{uv}
\end{align}
where $d\sigma$ is the measure on the given hyperplane inherited from $\rthree \times \rthree$.

Next, we point out an useful relationship between the mixed Lebesgue spaces and the $X^{s,b}$ spaces, which is based on the Strichartz estimates for the free Schr\"odinger equation. Namely, since we have $\|e^{-i t \p_x^2} f\|_{L^q_t L^r_x}\leq C_{Str.} \|f\|_{L^2}$, for all pairs $(q,r): 2\leq q,r\leq \infty$, $\f{2}{q}+\f{1}{r}=\f{1}{2}$, we may conclude the following estimate 
\begin{equation}
\label{a:1}
\|u\|_{L^q_t L^r_x}\leq C_{Str., \de} \|u\|_{X^{0,\f{1}{2}+\de}}.
\end{equation}
for all Strichartz pairs $(q,r) :\f{2}{q}+\f{1}{r}=\f{1}{2}$ and for all $\de>0$. 

Motivated from the expression \eqref{uv}, we introduce the bilinear $L^2$ multiplier norm.  For the rest of the section, we generally follow Tao's setup, \cite{Tao1}.\\

Let $\tau = (\tau_1, \tau_2,\tau_3) \in \rthree$, $\xi = (\xi_1, \xi_2, \xi_3) \in \rthree$.  Let $\Gamma$ be a hyperplane in $\rthree\times \rthree$ such that $\tau_1+\tau_2+\tau_3=0$ and $\xi_1+ \xi_2+\xi_3=0$, and let $d\sigma$ be the measure on $\Gamma$ inherited from $\rthree \times \rthree$.\\

Given a function~$m(\tau,\xi)$ defined on $\Gamma$, we define $||m||_{\mathcal{M}}$ to be the smallest constant~$C$ such that the following inequality holds:
\begin{equation}
\label{mnorm}
\left| \int_{\Gamma} m(\tau,\xi) u(\tau_1,\xi_1) v(\tau_2,\xi_2) w(\tau_3,\xi_3) \, d\sigma \right| 
\leq C \|u\|_{L^2_{\tau_1, \xi_1}} \|v\|_{L^2_{\tau_2,\xi_2}} \|w\|_{L^2_{\tau_3,\xi_3}}.
\end{equation}

From the definition \eqref{mnorm}, we claim the Comparison Principle.  That is, if $|m(\tau,\xi)| \leq M(\tau,\xi)$ for all $(\tau,\xi)\in \Gamma$, then $\|m\|_{\mathcal{M}} \leq \|M\|_{\mathcal{M}}$.\\

\subsection{Estimates on bilinear $L^2$ multiplier norms from \cite{Tao1}}

We introduce new notations which will be useful for working with these multipliers.  Let  
$$
\widetilde{u_{N,L}} (\tau,\xi) = \chi_{[N, 2N]} (|\xi|) \chi_{[L,2L]} ( |\tau-\xi^2|) \widetilde{u}(\tau,\xi),
$$   %$\widetilde{u_{N,\overline{L}}}(\tau,\xi) = \chi_{[N, 2N]} (|\xi|) \chi_{[L,2L]} ( |\tau+\xi^2|) \widetilde{u}(\tau,\xi)$.  
By using capitalized letters~$N,L$, we will always assume that $N,L$ are dyadic numbers, i.e. 
$N=2^j$, $L=2^l$ for some $j,l\in \zz$. \\

We will often encounter the following expression (for $\ve_j = \pm 1$ for $j=1,2,3$):
\begin{equation}\label{mult}
\|\chi_{[H,2H]} (|\ve_1 \xi_1^2+\ve_2 \xi_2^2 +\ve_3 \xi_3^2|) \prod_{j=1}^3 \chi_{[N_j, 2N_j]}(|\xi_j|) \chi_{[L_j,2L_j]} (|\tau -\ve_j \xi^2|) \|_{\mathcal{M}}.
\end{equation}
We will denote the above $L^2$ bilinear multiplier by $\chi^{(\pm,\pm,\pm)}$ (i.e. $\chi^{(+,+,-)}$ refers to the above expression with $(\ve_1,\ve_2,\ve_3) = (+1,+1,-1)$).\\

The following statement is Proposition 11.1 in \cite{Tao1}.
\begin{proposition}[$(+,+,+)$ case]
\label{prop:5} 
Let $H,N_1,N_2,N_3,L_1,L_2,L_3>0$ and $\ve_1=\ve_2=\ve_3 =+1$ in \eqref{mult}.  
If $N_{\max} \sim N_{med}$ and $H\sim N_{\max}^2$ and $L_{\max} \sim \max(H,L_{\med})$, then we have following cases 
\begin{enumerate}
\item In the exceptional case where $N_{\max} \sim N_{\min}$ and $L_{\max}\sim H$,
\begin{equation}
\label{ppp2}
\eqref{mult} \leq C L_{\min}^{\frac{1}{2}} L_{\med}^{\frac{1}{4}}
\end{equation}
\item In all the remaining cases (i.e. $N_{\min}<<N_{\max}$ or $L_{\max}\nsim H$), there is an absolute constant $C$, so that  
\begin{equation}\label{ppp1}
\eqref{mult} \leq C \frac{L_{\min}^{\frac{1}{2}} L_{\med}^{\frac{1}{2}}}{N_{\max}^{\frac{1}{2}}}
\end{equation}
\end{enumerate}
\end{proposition}
The following Proposition covers the other important case - namely, when not all 
$\ve_j, j=1,2,3$ match. Note that while almost all the cases below already appear in the work of Tao, \cite{Tao1}, there is the estimate \eqref{ppm2}, which is not stated explicitly\footnote{although it is implicit in the arguments in \cite{Tao1}} 
 (and shall be crucial for us in the sequel). 
\begin{proposition}[$(+,+,-)$ case]\label{ppm}
Let $H,N_1,N_2,N_3,L_1,L_2,L_3>0$ and $(\ve_1,\ve_2,\ve_3) =(+1,+1,-1)$ in \eqref{mult}.  If $N_{\max} \sim N_{med}$, $H\sim N_1N_2$ and $L_{\max} \sim \max(H,L_{\med})$, then 
there is an absolute constant $C$, so that 

\begin{enumerate}
\item For all cases, 
\begin{equation}\label{ppm1}
\eqref{mult} \leq C L_{\min}^{\frac{1}{2}} N_{\min}^{\frac{1}{2}}
\end{equation}

\item For all cases,
\begin{equation}\label{ppm2}
\eqref{mult} \leq C \min \left( \frac{L_1 L_3}{N_2}, \frac{L_2 L_3}{N_1}\right)^{\frac{1}{2}}
\end{equation}

\item If $L_3=L_{\max}$ 
\begin{enumerate}
\item and $N_1 \sim N_2 \sim N_3$ holds, 
\begin{equation}\label{ppm4}
\eqref{mult} \leq C L_{\min}^{\frac{1}{2}}L_{\med}^{\frac{1}{4}}.
\end{equation}

\item and $N_1\sim N_2 \sim N_3$ does not hold, 
\begin{equation}\label{ppm3}
\eqref{mult} \leq C \frac{L_{\min}^{\frac{1}{2}}L_{\med}^{\frac{1}{2}}}{N_{\max}^{\frac{1}{2}}}
\end{equation}

\end{enumerate}
\end{enumerate}
\end{proposition}

\begin{proof}
We will only prove \eqref{ppm2}.  For the others, we refer to \cite{Tao1}.\\

Define the following sets.

\begin{align*}
A_1 &= \{ (\tau,\xi)\in \rtwo: \xi \sim N_1, \tau-\xi^2 \sim L_1 \}\\
A_2 &= \{ (\tau,\xi)\in \rtwo: \xi \sim N_2, \tau-\xi^2 \sim L_2 \}\\
A_3 &= \{ (\tau,\xi)\in \rtwo: \xi \sim N_3, \tau+\xi^2 \sim L_3 \}
\end{align*}

To prove \eqref{ppm2}, let $N_j =N_{\min}$.  Define $R =\{ (\tau,\xi)\in \rtwo :  |\xi| \leq \varepsilon N_{\min}\}$.  Then we can find $m=O(1/\varepsilon)$ numbers $\xi_j^0 \sim N_j$, so that the sets of type $(0,\xi_j^0) + R$ covers the set $A$.  Then we can apply the Box Localization (Corollary 3.13 of \cite{Tao1}) so that 
\[
\eqref{mult} \leq C \| \prod_{k=1}^3 \chi_{A_k\cap[(0,\xi_k^0) + R]} (\tau_k,\xi_k)\|_{\mathcal{M}}
\]
for some $\xi_k^0 \in A_k$ so that $\xi_1^0 + \xi_2^0 + \xi_3^0 \leq \varepsilon N_{\min}$.  Denote $A_k^0 = A_k\cap[(0,\xi_k^0) +R]$.  Now by comparison principle and Corollary 3.10 of \cite{Tao1},
\begin{align*}
\eqref{mult} &\leq C \| \prod_{k=2}^3 \chi_{A_k^0} (\tau_k,\xi_k)\|_{\mathcal{M}}\\
	&\leq C |\{ (\tau_2,\xi_2) \in A_2^0: (\tau,\xi) - (\tau_2,\xi_2) \in A_3^0 \}|^{\frac{1}{2}}
\end{align*}
for some $(\tau,\xi) \in A_1 + 2R$.  Note that $\xi \sim N_1$ for $\varepsilon>0$ small.\\

We have $\tau_2 =\xi_2^2+O(L_2)$ and $\tau - \tau_2 = -(\xi-\xi_2)^2 +O(L_3)$.  First we can remove $\tau_2$ by restricting it to an interval of length at most $O(\min(L_2,L_3))$ for a fixed $\xi_2$.  Furthermore, these restrictions give $\xi_2^2 - (\xi -\xi_2)^2 = \tau + O(\max(L_2,L_3))$; that is $2\xi \xi_2 = \tau+\xi^2 +O(\max(L_2,L_3))$.  So

\[
\eqref{mult} \leq C [\min(L_2,L_3) |\{ \xi_2 \sim N_2 : \xi_2 = \frac{\tau +\xi^2}{2\xi} + O(\max(L_2,L_3)/\xi)\}| \, ]^{\frac{1}{2}}.
\]

Clearly, $\xi_2$ above is contained in an interval of length at most $O(\max(L_2,L_3)/N_1)$.  So we get the estimate
\[
\eqref{mult} \leq C \left(\frac{L_2L_3}{N_1}\right)^{\frac{1}{2}}.
\]

By reversing the role of $A_1$, $A_2$ and following the same arguments, we also obtain
\[
\eqref{mult} \leq C \left(\frac{L_1L_3}{N_2}\right)^{\frac{1}{2}}.
\]

This proves \eqref{ppm2}.
\end{proof}

We now apply Proposition \ref{prop:5} and Proposition \ref{ppm} to deduce some important bilinear estimates, which will be one of the main tools in the sequel. 
\begin{lemma}\label{estimates}
Let $T>0$ and $u,v\in \mathcal{S}((-T,T)\times\rone)$.  
Then for $\delta>0$, $k>0$, and some $C= C(\delta, T)$;
\begin{eqnarray}
 \label{gain1}
& & \|(u_{k} v_{\ll k})_{\sim k}\|_{L_{t,x}^2} \leq C\, 2^{-(\frac{1}{2}-\delta)k } 
\|u\|_{X^{0,\frac{1}{2}+\delta}} \|v\|_{X^{0,\frac{1}{2}+\delta}}
\\
\label{gain2}
 & &  \|(u_{k} v_{\sim k})_{\ll k}\|_{L_{t,x}^2} \leq C \,2^{-(\frac{1}{2}-\delta)k} \|u\|_{X^{0,\frac{1}{2}+\delta}} \|v\|_{X^{0,\frac{1}{2}+\delta}}\\
 \label{gain3}
& &   \|u_{k} v\|_{L_{t,x}^2} \leq C \|u\|_{X_T^{0,\frac{1}{2}+\delta}} 
\|v\|_{X_T^{0,\frac{1}{2}+\delta}}
\end{eqnarray}
In addition, there are the following estimates concerning the bilinear form $(u,v)\to u\bar{v}$ 
\begin{eqnarray}
\label{kkk1}
 & &  \|(u_k \overline{v_k})_k\|_{L^2_{t x}}  \leq C\, 2^{-(\frac{1}{2}-\delta)k}
 \|u\|_{X^{0,\frac{1}{2}+\delta}}\|v\|_{X^{0,\frac{1}{2}+\delta}}\\
 \label{kkk2}
  & & \|(u_k \overline{v_{\ll k}})_{k}\|_{L^2_{t x}}  \leq C\, 2^{-(\frac{1}{2}-\delta)k}\|u\|_{X^{0,\frac{1}{2}+\delta}}\|v\|_{X^{0,\frac{1}{2}+\delta}}\\
  \label{kkk3}
  & & 
   \| u_k \overline{v}\|_{L^2_{t x}}  \leq C
   \|u\|_{X_T^{0,\frac{1}{2}+\delta}}  \|v\|_{X_T^{0,\frac{1}{2}+\delta}}
\end{eqnarray}
\end{lemma}

\begin{proof}
We first dispense with the easy estimates \eqref{gain3} and \eqref{kkk3}. Indeed, taking into account the boundedness of $P_{\sim k}$ and $P_{<<k}$ on all $L^p$ spaces, we estimate both expressions by H\"older's and \eqref{a:1} 
$$
C T^{1/4} \|u\|_{L^8_T L^4_x}  \|v\|_{L^8_T L^4_x}  \leq 
C_\de T^{1/4} \|u\|_{X^{0, \f{1}{2}+\de}}  \|v\|_{X^{0, \f{1}{2}+\de}}, 
$$
since $q=8, r=4$ is a Strichartz pair. 

For the estimates \eqref{gain1} and  \eqref{gain2}, we  use Proposition~\ref{ppm}.  We use the partition of unity $\chi^{(+,+,+)}$, where $N_1, N_2,N_3$ indicates the respective frequencies of $u, v, uv$.  Denote by $\sum$, summation over $N_1, N_2, N_3, L_1, L_2, L_3 \geq 1$.  
Note that $N_{\max} \sim 2^k$ and the relation
 $L_{\max} \sim \max(L_{\med}, N_{\max}^2)$, which holds trivially 
 by the constraints, see \cite{Tao1}. 

For \eqref{gain1}, we apply \eqref{ppp1} to obtain
\begin{align*}
 \|(u_{k} v_{\ll k})_{\sim k}\|_{L_{t,x}^2} &= \sup_{\|w\|_{L^2_{t,x}}=1}\left| \int_{\rone\times \rone} u_k(t,x) v_{\ll k}(t,x) w_{\sim k}(t,x) \, dt\, dx\right|\\
 	&\leq C \sum \frac{L_{\min}^{\frac{1}{2}} L_{\med}^{\frac{1}{2}}}{N_{\max}^{\frac{1}{2}}} \|u_k\|_{L^2_{t,x}} \|v_{\ll k}\|_{L^2_{t,x}}\\
 	&\leq C \sum \frac{1}{L_1^{\frac{1}{2}+\delta} L_2^{\frac{1}{2}+\delta} } \frac{L_{\min}^{\frac{1}{2}} L_{\med}^{\frac{1}{2}}}{N_{\max}^{\frac{1}{2}}} \|u\|_{X^{0,\frac{1}{2}+\delta}} \|v\|_{X^{0,\frac{1}{2}+\delta}} \\
 	&\leq C \sum \frac{1}{L_{\med}^{\delta} N_{\max}^{\frac{1}{2}} }  \|u\|_{X^{0,\frac{1}{2}+\delta}} \|v\|_{X^{0,\frac{1}{2}+\delta}} \\
	&\leq C_{\delta} 2^{\left(-\frac{1}{2}+\delta\right)k} \|u\|_{X^{0,\frac{1}{2}+\delta}} \|v\|_{X^{0,\frac{1}{2}+\delta}}.
\end{align*}
\eqref{gain2} is estimated exactly the same way as \eqref{gain1}.

To prove \eqref{kkk1}and \eqref{kkk2}, we use Proposition~\ref{ppm}.  We use the partition of unity $\chi^{(+,+,-)}$, where $N_1, N_2,N_3$ indicates the respective frequencies of $u,  u\overline{v}, \overline{v}$.  Denote by $\sum$, summation over $N_1, N_2, N_3, L_1, L_2, L_3 \geq 1$.  Note $L_{\max} \sim \max(L_{\med}, N_1 N_2)$ and $N_{1} \sim 2^k$.\\

For both \eqref{kkk1} and \eqref{kkk2}, $N_2\sim N_{\max}\sim 2^k$.  Since the calculations will be almost identical, we will only prove \eqref{kkk1} here.  We apply \eqref{ppm2} to obtain  
\begin{align*}
 \|(u_k \overline{v_{\sim k}})_{\sim k}\|_{L^2} &= \sup_{\|w\|_{L^2}=1} \left|\int_{\rone\times \rone} u_k (t,x) \overline{v_{\sim k}} (t,x) w_{\sim k} (t,x)\, dt\, dx \right| \\
 	&\leq C \sum  \frac{L_1^{\frac{1}{2}} L_3^{\frac{1}{2}}}{N_2^{\frac{1}{2}}} \|u_k\|_{L^2_{t,x}} \|\overline{v_{\sim k}}\|_{L^2_{t,x}}\\ 
 	&\leq C 2^{\left(-\frac{1}{2}+\delta\right)k}  \|u\|_{X^{0,\frac{1}{2}+\delta}} \|v\|_{X_{\tau=\xi^2}^{0,\frac{1}{2}+\delta}}
 \end{align*}
 	\end{proof}
 	We now provide a technical corollary, which allows us to put $\|v\|_{X^{0,\f{1}{2}-\de}}$ norms on the right hand sides of \eqref{kkk1}, \eqref{kkk2} and \eqref{kkk3}, at the expense of slightly less gain in the power of $2^k$. 
 	\begin{corollary}
 	\label{cor:a1} 
 	With the assumptions in Lemma \ref{estimates}, we have 
 	\begin{eqnarray}
\label{kkkk1}
 & &  \|(u_k \overline{v_k})_k\|_{L^2_{t x}}  \leq C\, 2^{-k(\frac{1}{2}-5\delta)}
 \|u\|_{X^{0,\frac{1}{2}+\delta}}\|v\|_{X^{0,\frac{1}{2}-\delta}}\\
 \label{kkkk2}
  & & \|(u_k \overline{v_{\ll k}})_{k}\|_{L^2_{t x}}  \leq C\, 2^{-k(\frac{1}{2}-5\delta)}\|u\|_{X^{0,\frac{1}{2}+\delta}}\|v\|_{X^{0,\frac{1}{2}-\delta}}\\
  \label{kkkk3}
  & & 
   \|u_k \overline{v}\|_{L^2_{t x}}  \leq C 2^{2\de k} 
   \|u\|_{X_T^{0,\frac{1}{2}+\delta}}  \|v\|_{X_T^{0,\frac{1}{2}-\delta}}
\end{eqnarray}
 	\end{corollary} 
 	\begin{proof}
 	We will show only \eqref{kkkk1}, the others follow similar route. Indeed, we use a combination of H\"olders with the Sobolev embedding $\|u_k\|_{L^\infty_x}\leq C 2^{k/2} \|u_k\|_{L^2_x}$ to obtain the following estimate 
 	\begin{eqnarray*}
  \|(u_k \overline{v_k})_k\|_{L^2_{t x}} &\leq &  \|u_k\|_{L^\infty_{t x}} \|v_k\|_{L^2_{t x}} \leq 
 	C 2^{k/2} \|u_k\|_{L^\infty_x L^2_x} \|v\|_{X^{0,0}}\leq \\ 
 	&\leq & C_\de 2^{k/2} \|u\|_{X^{0,\f{1}{2}+\de}}  \|v\|_{X^{0,0}}.
 	\end{eqnarray*}
 	For a fixed function $u$, we are set to use complex interpolation between this and \eqref{kkk1}. Noting that $[X^{0,\f{1}{2}+\de}, X^{0,0}]_{4\de}=X^{0,\f{1}{2}-\de-4\de^2}$, we conclude 
 	$$
 	\|(u_k \overline{v_k})_k\|_{L^2_{t x}}\lesssim  2^{-k(\f{1}{2}-5\de+4\de^2)} 
 	\|u\|_{X^{0,\f{1}{2}+\de}} \|v\|_{X^{0,\f{1}{2}-\de-4\de^2}}\lesssim 
 	 2^{-k(\f{1}{2}-5\de)}\|u\|_{X^{0,\f{1}{2}+\de}} \|v\|_{X^{0,\f{1}{2}-\de}}
 	$$
 	For the proof of \eqref{kkkk3}, we interpolate between \eqref{kkk3} and the estimate 
 	$$
 	\|u_k\bar{v}\|_{L^2_{t x}}\leq C \|u_k\|_{L^\infty_{t x}} \|v\|_{L^2_{t x}}\leq C_\de 2^{k/2} 
 	\|u_k\|_{X^{0, \f{1}{2}+\de}} \|v\|_{X^{0,0}}
 	$$ 
 	\end{proof}
 	
 	The next lemma is new and addresses one situation in the (generally unfavorable) case 
\eqref{gain3}, where one still can get a gain of almost half derivative. 
\begin{lemma}
We denote $\widehat{u^+}(\xi)=\widehat{u}(\xi)\chi_{(0,\infty)}(\xi), 
\widehat{u^-}(\xi)=\widehat{u}(\xi)\chi_{(-\infty,0)}(\xi)$. For all $k>0$,   
\begin{equation}\label{plusminus}
||(u^+_{\sim k} v^-_{\sim k})_{\sim k}||_{L^2} \leq C \, 2^{-(\frac{1}{2}-\delta)k}||u||_{X^{0,\frac{1}{2}+\delta}}||v||_{X^{0,\frac{1}{2}+\delta}},
\end{equation}
for some absolute constant $C$. 
\end{lemma}

\begin{proof}
We present  the argument for $||(u^+_k v^-_k)^-_k||_{L^2}$.  The proof for the other case 
$||(u^+_k v^-_k)^+_k||_{L^2}$ is analogous.\\

First we define the following sets.
\begin{align*}
A &:= \{ (\tau,\xi)\vert \xi>0, \xi \sim 2^{k}, |\tau-\xi^2| \sim L_1 \}\\
B &:= \{ (\tau,\xi)\vert \xi <0, \xi \sim 2^{k}, |\tau - \xi^2| \sim L_2 \}\\
C &:= \{ (\tau, \xi) \vert \xi <0, \xi \sim 2^k \}
\end{align*}
Then we need to show 
\begin{equation}\label{chi3}
||\chi_A (\tau_1, \xi_1) \chi_B(\tau_2, \xi_2) \chi_C (\tau_1 + \tau_2, \xi_1+\xi_2) ||_{\mathcal{M}} \lesssim \frac{L_1^{\frac{1}{2}} L_2^{\frac{1}{2}}}{2^{\frac{k}{2}}}.
\end{equation}
Note that if $\max (L_1, L_2) \gtrsim 2^{2k}$, then \eqref{ppp2} gives us the desired statement.  Otherwise, we have $\max(L_1, L_2) \ll 2^{2k}$.\\

For some $\varepsilon>0$ small, we partition $A$ (similarly $B$) into $m =O(1/\varepsilon)$ subsets $A_1, \cdots, A_m$ so that the diameter of $A_j$ in $\xi$~variable (similarly $B_j$) is less than $\varepsilon 2^k$ for all $1\leq j \leq m$.  Then, removing the terms when $\chi_{A_i} \chi_{B_j} \chi_{C}=0$, we can omit $\chi_C$ from the expression \eqref{chi3}.\\

Applying the comparision principle and Corollary 3.10 of \cite{Tao1}, the left side of \eqref{chi3} is bounded by
\begin{equation}\label{bound3}
\sum_{i,j=1}^m \left|\{ (\xi_1,\tau_1) \in A_i : \, (\tau,\xi) - (\tau_1,\xi_1) \in B_j \}\right|^{\frac{1}{2}}
\end{equation}
where $\tau, \xi$ are fixed.  Note that $(\tau,\xi)\in C$ or at most distance~$O(\varepsilon 2^k)$ from $C$, so $\xi\sim -2^k$.  Writing out the condition of the set gives $\tau_1 = \xi_1^2 +O(L_1)$ and $\tau -\tau_1 = (\xi-\xi_1)^2 +O(L_2)$.  So, for a fixed $\xi_1$, the $\tau_1$ must be in an interval of length
$O(\min (L_1,L_2))$. Then \eqref{bound3} is bounded by
\[
\sum_{i,j=1}^m \min(L_1,L_2)^{\frac{1}{2}} \left|\{ \xi_1 >0, \xi_1 \sim 2^k: \xi_1^2 + (\xi-\xi_1)^2 = \tau + O(\max(L_1,L_2))\} \right|^{\frac{1}{2}}
\]
We can write $\displaystyle \xi_1^2 + (\xi-\xi_1)^2 = \frac{\xi^2 +(2\xi_1-\xi)^2}{2}$.  So the condition given above can be written as
\[
(\xi_1-\frac{\xi}{2})^2 = C_{\tau,\xi} +O(\max (L_1, L_2))
\]

where $C_{\tau,\xi} := \frac{2\tau - \xi^2}{4}$.  Since $\xi_1$ and $\xi$ have the opposite sign, the left hand side of the above is $\sim 2^{2k}$.  On the other hand,  $\max (L_1, L_2) \ll 2^{2k}$, so $C_{\tau,\xi} \sim 2^{2k}$.  Then we have
\begin{align*}
|\xi_1 - (\frac{\xi}{2} + \sqrt{C_{\tau,\xi}})| &= \sqrt{C_{\tau,\xi} + O(\max (L_1, L_2))}- \sqrt{C_{\tau,\xi}}\\
 	&= \frac{O(\max (L_1, L_2))}{\sqrt{C_{\tau,\xi} + O(\max (L_1, L_2))}+ \sqrt{C_{\tau,\xi}}}\\
 	&\lesssim \frac{O(\max (L_1, L_2))}{2^k}.
\end{align*}
So $\xi_1$ must be contained in an interval of length $\displaystyle \frac{O(\max (L_1, L_2))}{2^k}.$  Using this interval in \eqref{bound3} gives the desired estimate \eqref{chi3}. 
\end{proof}

\section{Proof of Theorem \ref{theo:1}}
\label{sec:s3}
We first perform a change of variables, which transforms the problem \eqref{1} into a problem with data in $L^2$. Namely, let 
$v: u=\naba v$. A quick calculation then shows that 
\eqref{1} becomes 
\begin{equation}
\label{2}
\left| 
\begin{array}{l}
v_t+ i \p_x^2 v= \nab^{\be-\al} [\naba v   \naba v]: \qq (t,x)\in \rone \times \rone\\
v(0,x)=\nabam g=:f\in L^2(\rone).
\end{array}\right.
\end{equation}
Thus, we need to study the well-posedness of \eqref{2} in the $L^2$ setting.  
\subsection{Setting up the normal forms} 
Introduce  $G(u,v):= \nab^{\be-\al} [\naba u    \naba v]$, so that the nonlinearity in \eqref{2} is of the form $G(v, v)$, note $G(u,v)=G(v,u)$.  Observe that  the bilinear  form 
 $G$  may be written as follows 
$$
G(u,v)(x)= \f{1}{4\pi^2} \int \f{\langle \xi\rangle^\al \langle\eta\rangle^\al }{\langle\xi+\eta\rangle^{\al-\be}}
 \widehat{u}(\xi) \widehat{v}(\eta) e^{i(\xi+\eta)x} d\xi 
d\eta.
$$
We now decompose the form $G(v,v)$ as follows 
\begin{eqnarray*}
G(v,v) &=&  G(v_{\leq 0},v)+G(v_{> 0}, v)= 
G(v_{\leq 0},v)+G(v_{>0}, v_{\leq 0})+ G(v_{>0}, v_{>0})= \\
&=&  G(v_{\leq 0},(Id+P_{>0})v)+ G(v_{>0}, v_{>0}).
\end{eqnarray*}
Next, we perform a change of variables $v\to z$,  $v=e^{-i t \p_x^2} f+z$. Clearly, $z(0)=0$ and  
\begin{eqnarray*}
z_t+i \p_x^2 z &=& G([e^{-i t \p_x^2} f+z]_{\leq 0}, (Id+P_{> 0})[e^{-i t \p_x^2} f+z])+ \\
& &+ 
G([e^{-i t \p_x^2} f+z]_{>0}, [e^{-i t \p_x^2} f+z]_{>0}) 
\end{eqnarray*}
Clearly, a lot of terms are generated by this transformation. We comment now on the form 
of various terms (especially the least favorable ones!), since this will influence 
our normal form analysis. 

Heuristically, if we expect   the $z$ term to be smoother, then 
the least smooth term is expected to be $G(e^{-i t \p_x^2} f_{>0}, 
e^{-i t \p_x^2} f_{>0})$. Indeed, there are $\al$ derivatives acting on each of the two 
entries (which  free solutions and hence, in general, no better than $L^2_x$ smooth) and $\be-\al$ derivatives acting on the product itself\footnote{In fact this $\be-\al$ derivatives on the product may not be of much help in ``high-high to low interaction scenario}.  
Thus, if we manage to build a 
smoother function $h$, which solves 
\begin{equation}
\label{30}
(\p_t+i\p_x^2) h= G(e^{-i t \p_x^2} f_{>0}, 
e^{-i t \p_x^2} f_{>0}),
\end{equation}
one would be compelled to change variables again, $z\to w$, where 
 $z=h+w$. 
Define   a bilinear operator $T$ 
$$
T(u,v)(x)= \f{1}{8\pi^2 i} \int 
\f{\langle\xi\rangle^\al\langle\eta\rangle^\al}{\langle\xi+\eta\rangle^{\al-\be}}
\f{1}{ \xi \eta} \widehat{u_{>0}}(\xi) \widehat{v_{>0}}(\eta) e^{i(\xi+\eta)x} d\xi 
d\eta.
$$
It is easy to check that for a pair of functions $u(t,x), v(t,x) \in \mathcal{S}(\rone\times \rone)$, 
\begin{equation}
\label{8} 
(\p_t+i \p_x^2) T(u,v)=T((\p_t+i \p_x^2) u, v)+
T(  u, (\p_t+i \p_x^2)v)+G(u_{>0},v_{>0}).
\end{equation} 
This last identity tells us that 
$$
(\p_t+i \p_x^2) T(e^{-i t \p_x^2} f, e^{-i t \p_x^2} f)= 
G(e^{-i t \p_x^2} f_{>0}, e^{-i t \p_x^2} f_{>0})
$$
which provides an explicit solution\footnote{Note that while 
the solution $h(t)$ of the Schr\"odinger equation  \eqref{30} is not unique, it is completely determined by its value $h(0)$}of \eqref{30}. Hence, set 
$$
h:= T(e^{-i t \p_x^2} f, e^{-i t \p_x^2} f). 
$$
We now change variables $z=h+w=T(e^{-i t \p_x^2} f,e^{-i t \p_x^2} f)+w$, 
whence we get the following equation for $w$ 
\begin{equation}
\label{a:2}
\begin{array}{ll}
w_t+i \p_x^2 w &= G([e^{-i t \p_x^2} f+h+w]_{\leq 0}, (Id+P_{> 0})[e^{-i t \p_x^2} f+h+w])+\\
 &+ 2 G((h+w)_{>0},e^{-i t \p_x^2} f_{>0})+ G((h+w)_{>0}, (h+w)_{>0})
 \end{array}
\end{equation}
 Note also that since $z(0)=0$, it follows that the Schr\"odinger equation for $w$ is supplemented by the following initial condition:  $w(0)=-h(0)=-T(f,f)$.  We have now prepared ourselves to close the argument in the $w$ variable. More precisely, 
 the proof of Theorem \ref{theo:1} reduces to establishing the 
 local well-posedness of the Schr\"odinger equation \eqref{a:2} in an appropriate function space. 
 
 Fix $1>>\de>0$.  Consider the spaces 
 \begin{eqnarray*}
 & & \cx=X^{\al-\f{1}{2},\f{1}{2}+\de}, \\ 
 & & \ch= L^\infty_t H^{\f{1}{2}}_x \cap X^{1-\de, \de}, 
 \end{eqnarray*} 
Our strategy will be to show that the fixed point argument for $w$ closes in the space $\cx$, given that $f\in L^2$, $h\in \ch$ and where we will occasionally need to use the particular form $h=T(e^{-i t \p_x^2} f, e^{-i t \p_x^2} f)$. 

\subsection{Estimates on the normal form}
\label{sec:s32}
In this section, we show the required smoothness of the normal form $h$, 
namely $h\in \ch$.  This will be done in two steps - in  Lemma \ref{tmap1} and Lemma \ref{tmap3}. 
 \begin{lemma} \label{tmap1}
 $T: L^2(\mathbb{R}) \times L^2 (\mathbb{R}) \to H^{\frac{1}{2}}(\mathbb{R})$ continuously.
 \end{lemma}
 \begin{proof} 
 Let $u,v\in \mathcal{S}(\mathbb{R})$.  Then 
 \begin{equation}\label{est}
 \|T(u,v)\|_{H^{\frac{1}{2}}} \leq \sum_{k\geq l+3}  \|T(u_k,v_l)\|_{H^{\frac{1}{2}}} + 
 \sum_{|k-l|< 3} \|T(u_k,v_l)\|_{H^{\frac{1}{2}}} = I_1 + I_2
 \end{equation}
 where $k,l >0$.  Regarding the first sum in (\ref{est}), we apply H\"older's and then 
 the Sobolev embedding $\|u_l\|_{L^\infty}\leq C 2^{l/2} \|u_l\|_{L^2}$. We get 
\begin{align*}
I_1 &\leq C \sum_{l>0}\sum_{k\geq l} 
2^{\frac{k}{2}}\frac{2^{\alpha k +\alpha l}}{ 2^{(\alpha-\beta) k + k+ l}} \|u_k v_l\|_{L^2} \leq C \sum_{l>0}\sum_{k\geq l} 2^{(\beta-\frac{1}{2}) k +(\alpha -1/2) l}\|u_k\|_{L^2}  
	\|v_l\|_{L^2}\\
	&\leq C \|u\|_{L^2} \|v\|_{L^2} \sum_{l>0} 2^{(\alpha+\beta -1)l} \leq C \|u\|_{L^2} \|v\|_{L^2}
\end{align*}

Regarding the second sum in (\ref{est}), 
\begin{align*}
I_2 &\leq C \sum_{k>0} \sum_{m\leq k+2} 2^{\frac{1}{2}m} \frac{ 2^{2\alpha k}}{2^{(\alpha-\beta) m + 2k }} ||P_m(u_k v_k)||_{L^2} \leq C \sum_{k>0} \sum_{m\leq k+2} 2^{(\frac{1}{2}-\alpha+\beta)m +(2\alpha-2) k} 2^{\frac{m}{2}}||u_k v_k||_{L^1}\\
   &\leq C \|u\|_{L^2} \|v\|_{L^2}\sum_{k>0} 2^{(\alpha+\beta-1) k}  \leq C  \|u\|_{L^2} \|v\|_{L^2}. 
\end{align*}
\end{proof}  
The next lemma provides a different type of estimate, namely that if we measure \\ 
$T(e^{-i t \p_x^2} f, e^{-i t \p_x^2} f)$ 
in averages $L^2_{t x}$ sense, we actually get a full spatial derivative gain. More precisely, 
\begin{lemma}\label{tmap3}
Let $u,v \in X^{0,\frac{1}{2}+\delta}_{\tau=\xi^2}$, then for $0\leq \delta < 1-\alpha-\beta$, 

\[
\|T(u, v)\|_{X^{ 1-\delta, \delta}} \leq C_{\delta} \|u\|_{X^{0,\frac{1}{2}+\delta}} \|v\|_{X^{0,\frac{1}{2}+\delta}}.
\]
\end{lemma}

\begin{proof}
By using the partition of unity $\chi^{(+, +, -)}$, we can localize spacial and time frequencies to their respective indices.  (Here we localize $u, v, uv$ respectively to $N_1, N_2, N_3$.)  Also denote by the symbol~$\sum$ to be the summation over $N_1, N_2, N_3, L_1, L_2, L_3 \geq 1$. 
We have 
\begin{align}
\|T(u,v)\|_{X^{ 1-\delta, \delta}} &\leq C \sum \frac{N_3^{1-\alpha+\beta-\delta}}{N_1^{1-\alpha} N_2^{1-\alpha}} \|u \, v\|_{X^{0,\delta}} \notag\\
	&\leq C \sum \frac{N_3^{1-\alpha+\beta-\delta}}{N_1^{1-\alpha} N_2^{1-\alpha}} \sup_{\|w\|_{X^{0,-\delta}_{\tau=-\xi^2}}=1} \left| \int_{\rone\times \rone} u(t,x) v(t,x) w(t,x)\, dt\, dx\right| \notag\\
	&\leq C 
	\sup_{\|w\|_{X^{0,-\delta}_{\tau=-\xi^2}}=1}\sum \frac{N_3^{1-\alpha+\beta-\delta}}{N_1^{1-\alpha} N_2^{1-\alpha}}   
	\|\chi^{(+,+,-)}\|_{\mathcal{M}}   \|u\|_{L^2_{t,x}} \|v\|_{L^2_{t,x}}\|w\|_{L^2_{t,x}} \\
	&\leq C \|u\|_{X^{0,\frac{1}{2}+\delta}} \|v\|_{X^{0,\frac{1}{2}+\delta}}  
	\sum \frac{N_3^{1-\alpha+\beta-\delta}L_3^{\delta}}{N_1^{1-\alpha} N_2^{1-\alpha}L_1^{\frac{1}{2}+\delta} L_2^{\frac{1}{2}+\delta}}
\|\chi^{(+,+,-)}\|_{\mathcal{M}}  \label{t1}
\end{align}

We refer to Proposition~\ref{ppm}.  If $\max(L_1,L_2) = L_{\max}$, we use that $\max(N_1,N_2) \gtrsim N_3$ to simplify \eqref{t1} and then apply the multiplier 
bound~\eqref{ppm1}.  We estimate the sum in \eqref{t1} by
\begin{align*}
 \sum \frac{L_3^{\delta}N_3^{\beta -\delta}}{L_1^{\frac{1}{2}+\delta} L_2^{\frac{1}{2}+\delta}}
L_{\min}^{\frac{1}{2}} N_{\min}^{\frac{1}{2}}  \leq  C  
	\sum \frac{N_{\min}^{\frac{1}{2}-\delta} N_{\max}^{\beta}}{L_{\max}^{\frac{1}{2}}}  \leq C_{\delta}
\end{align*}

If $L_3=L_{\max}$ and $N_1 \sim N_2 \sim N_3 \sim N$, then we can 
assume $L_3 \sim N^2$, so applying \eqref{ppm4} yields the estimate 
\begin{align*}
 \sum \frac{N^{1-\alpha+\beta-\delta}L_3^{\delta}}{N^{1-\alpha} N^{1-\alpha}L_1^{\frac{1}{2}+\delta} L_2^{\frac{1}{2}+\delta}}
L_{\min}^{\frac{1}{2}} L_{\med}^{\frac{1}{4}}   \leq C 
\sum \frac{1}{ N^{1-\alpha-\beta -\delta}}
 \leq C_{\delta},
 \end{align*}
 provided $0<\de<<1-\al-\be$. \\
 The remaining case is when  $L_3 = L_{\max} \sim N_1 N_2$ with the bound~\eqref{ppm3}.  We have 
$$
 \sum \frac{N_3^{1-\alpha+\beta-\delta}L_3^{\delta}}{N_1^{1-\alpha} N_2^{1-\alpha}L_1^{\frac{1}{2}+\delta} L_2^{\frac{1}{2}+\delta}}
\frac{L_{\min}^\frac{1}{2}L_{\med}^{\frac{1}{2}}}{N_{\max}^{\frac{1}{2}}}  \leq C \sum 
N_{\max}^{\beta -\frac{1}{2}+\delta}  \leq C.
$$
 	
 	\end{proof}

\subsection{Some bilinear and trilinear estimates involving typical right-hand sides}
\label{sec:s33}

Let us start with few words regarding strategy. All the terms (with an exception of one single term) in the right hand-side of \eqref{a:2}, which contain at least one 
smooth term (i.e.  in the form $u_{\leq 0}$) will be dealt with by relatively simple arguments, mainly based 
on Lemma \ref{estimates}. For all other terms, we shall need specific (bilinear and trilinear) 
 estimates, which handle different type of configurations (i.e. $h$ and $w$, $h$ and $e^{-i \p_x^2} f$) on the right-hand side of \eqref{a:2}. 
These multilinear estimates are presented below. We also attempt to indicate the relevancy of each such estimate, before the statement of each Lemma. 

The next lemma is useful, when one deals with terms in the form 
$G(w, e^{-i \p_x^2} f)$ and $G(w,w)$  on the right hand side of \eqref{a:2}. 

\begin{lemma}
\label{le:5}
Let $z$ satisfy $(\p_t+i\p_x^2) z = G(u_{>0}, v_{>0})$ with $z(0,x) \equiv 0$, and let $0<10\de <1-\al-\be$. Then
\[
\|z\|_{X_T^{\al-\f{1}{2}, \frac{1}{2}+\delta}} \leq C_{T,\delta} \|u\|_{X^{\al-\f{1}{2}, \frac{1}{2}+\delta}} \|v\|_{X^{0,\frac{1}{2}+\delta}}.
\]
\end{lemma}

\begin{proof}
By using the partition of unity $\chi^{(+, +, -)}$, we can localize spacial and time frequencies to their respective indices.  (Here we localize $u, v, uv$ respectively to $N_1, N_2, N_3$.)  Also denote by the symbol~$\sum$ to be the summation over $N_1, N_2, N_3, L_1, L_2, L_3 \geq 1$. \\
  
Applying Proposition~\ref{xsb}, we obtain
\begin{align}
\|z\|_{X_T^{\al- \f{1}{2}, \frac{1}{2}+\delta}} &\leq C_{T,\delta} \|G(u_{>0}, v_{>0})\|_{X^{\al - \frac{1}{2}, -\frac{1}{2}+\delta}}\notag\\
	&\leq C_{T,\delta} \sum \frac{N_1^{\alpha} N_2^{\alpha}}{N_3^{\frac{1}{2}-\be}} \|u\, v\|_{X^{0,-\frac{1}{2}+\delta}}\notag\\
	&\leq C_{T,\delta} \sum \frac{N_1^{\alpha} N_2^{\alpha}}{N_3^{ \frac{1}{2}-\be}}  \sup_{\|w\|_{X^{0, \frac{1}{2}-\delta}_{\tau=-\xi^2}}=1} \left| \int_{\rone\times \rone} u (t,x) v (t,x) w (t,x) \,dt\,dx\right|\notag\\
	&\leq C_{T,\delta} \sum \frac{N_1^{\alpha} N_2^{\alpha}}{N_3^{\f{1}{2}-\be}}  \sup_{\|w\|_{X^{0, \frac{1}{2}-\delta}_{\tau=-\xi^2}}=1} \|\chi^{(+,+,-)}\|_{\mathcal{M}} \|u\|_{L^2_{t,x}} \|v\|_{L^2_{t,x}} \|w\|_{L^2_{t,x}}\notag\\
	&\leq C_{T,\delta} \|u\|_{X^{\al-\f{1}{2},\frac{1}{2}+\delta}} \|v\|_{X^{0,\frac{1}{2}+\delta}} \sum \frac{N_1^{\f{1}{2}} N_2^{\alpha}}{N_3^{\frac{1}{2}-\be} L_1^{\frac{1}{2}+\delta} L_2^{\frac{1}{2}+\delta} L_3^{\frac{1}{2}-\delta}} \|\chi^{(+,+,-)}\|_{\mathcal{M}} \label{w1}
\end{align}

Now we refer to Proposition~\ref{ppm} for appropriate bounds.  If we have $L_{\max} \sim L_{\med} \gg N_1 N_2$,  we apply \eqref{ppm1}.  Then the sum in \eqref{w1} is estimated by
\begin{align*}
 \sum \frac{N_1^{\f{1}{2}} N_2^{\alpha}}{N_3^{\frac{1}{2}-\be} L_1^{\frac{1}{2}+\delta} L_2^{\frac{1}{2}+\delta} L_3^{\frac{1}{2}-\delta}}  N_{\min}^{\frac{1}{2}} L_{\min}^{\frac{1}{2}}
	&\leq \sum \frac{N_1^{\f{1}{2}} N_2^{\alpha} N_{\min}^{\be}}{L_{\max} } \leq \sum \frac{1}{L_{\max}^{1-\alpha-\beta}} \leq C_{\alpha,\beta}
\end{align*}	

Otherwise, we may assume that $L_{\max} \sim N_1 N_2$.\\

If $L_1=L_{\max}$, we apply $\displaystyle \|\chi^{(+,+,-)}\|_{\mathcal{M}} \leq C \frac{L_2^{\frac{1}{2}} L_3^{\frac{1}{2}}}{N_1^{\frac{1}{2}}}$.  If $L_2 = L_{\max}$, then we can use $\displaystyle \|\chi^{(+,+,-)}\|_{\mathcal{M}} \leq C \frac{L_1^{\frac{1}{2}} L_3^{\frac{1}{2}}}{N_2^{\frac{1}{2}}}$.\\

Both cases will work out similarly.  For instance, if $L_1 = L_{\max}$, then the sum in \eqref{w1} is estimated
\begin{align*}
 \sum \frac{N_1^{\f{1}{2}} N_2^{\alpha}}{N_3^{\frac{1}{2}-\be} L_1^{\frac{1}{2}+\delta} L_2^{\frac{1}{2}+\delta} L_3^{\frac{1}{2}-\delta}}  \frac{L_2^{\frac{1}{2}} L_3^{\frac{1}{2}}}{N_1^{\frac{1}{2}}}
	&\leq  \sum \frac{ N_2^{\alpha}}{N_3^{\frac{1}{2}-\be} L_{\max}^{\f{1}{2}+\delta}}  \leq \sum \frac{1}{L_{\max}^{\delta}} = C_{\de}.
\end{align*}

If $L_3 = L_{\max}$ and $N_1\sim N_2 \sim N_3 \sim N$, then we can estimate the sum in \eqref{w1} via \eqref{ppm4}.
\begin{align*}
 \sum \frac{ N^{\alpha+\beta}}{L_{\max}^{\frac{1}{2}-\delta}} 
	&\leq \sum N^{\alpha+\beta -1 +2\delta} \leq C_{\al,\be, \de}.
\end{align*}

Otherwise, \eqref{ppm3} can be used to estimate the sum in \eqref{w1}
\begin{align*}
\sum \frac{N_1^{\f{1}{2}} N_2^{\alpha-\f{1}{2}}}{N_3^{\f{1}{2}-\be} L_{\max}^{\frac{1}{2}-\delta}} 	&\leq  \sum \frac{1}{ L_{\max}^{\delta}} \leq C_{\de}.
\end{align*}
\end{proof}
The next lemma deals with right-hand sides of the form $G(h,h)$. 

\begin{lemma}
\label{le:6}
Let $z$ satisfy $(\p_t+i\p_x^2) z = G(u_{>0}, v_{>0})$ with $z(0,x) \equiv 0$.  Then for $0<\delta <1-\alpha-\beta$,
\[
\|z\|_{X_T^{\al-\f{1}{2}, \frac{1}{2}+\delta}} \leq C_{T,\delta} 
\|u\|_{X^{1-\delta, \delta}} \|v\|_{X^{1-\delta,\delta}}.
\]
\end{lemma}

\begin{proof}
By using the partition of unity $\chi_{(+, +, -)}$, we can localize spacial and time frequencies to their respective indices.  (Here we localize $u, v, uv$ respectively to $N_1, N_2, N_3$.)  Also denote by the symbol~$\sum$ to be the summation over $N_1, N_2, N_3, L_1, L_2, L_3 \geq 1$. \\
  
Applying Proposition~\ref{xsb}, we obtain
\begin{align}
\|z\|_{X_T^{\al-\frac{1}{2}, \frac{1}{2}+\delta}} &\leq C_{T,\delta} \|G(u_{>0}, v_{>0})\|_{X^{\al-\f{1}{2}, -\frac{1}{2}+\delta}}\notag\\
	&\leq C_{T,\delta} \sum \frac{N_1^{\alpha} N_2^{\alpha}}{N_3^{\frac{1}{2}-\be}} \|u\, v\|_{X^{0,-\frac{1}{2}+\delta}}\notag\\
	&\leq C_{T,\delta} \sum \frac{N_1^{\alpha} N_2^{\alpha}}{N_3^{\frac{1}{2}-\be}}  \sup_{\|w\|_{X^{0, \frac{1}{2}-\delta}_{\tau=-\xi^2}}=1} \left| \int_{\rone\times \rone} u(t,x) v(t,x) w (t,x)\, dt\,dx\right|\notag\\
	&\leq C_{T,\delta} \sum \frac{N_1^{\alpha} N_2^{\alpha}}{N_3^{\frac{1}{2}-\be}}  \sup_{\|w\|_{X^{0, \frac{1}{2}-\delta}_{\tau=-\xi^2}}=1} \|\chi^{(+,+,-)}\|_{\mathcal{M}} \|u\|_{L^2_{t,x}} \|v\|_{L^2_{t,x}} \|w\|_{L^2_{t,x}}\notag\\
	&\leq C_{T,\delta} \|u\|_{X^{1-\delta,\delta}} \|v\|_{X^{1-\delta,\delta}} \sum \frac{1}{N_1^{1-\alpha} N_2^{1-\alpha} N_3^{\frac{1}{2}-\be} L_1^{\delta} L_2^{\delta} L_3^{\frac{1}{2}-\delta}}  \|\chi^{(+,+,-)}\|_{\mathcal{M}}\label{w2}
\end{align}

We refer to Proposition~\ref{ppm} for appropriate bounds.  We apply the bound~\eqref{ppm1} to the multiplier to estimate the sum in \eqref{w2} by
\begin{align*}
\sum  \frac{N_{\min}^{\frac{1}{2}} }{N_1^{1-\alpha} N_2^{1-\alpha} N_3^{ \frac{1}{2}-\be} L_{\max}^{\delta} } 	&\leq \sum  \frac{1}{ L_{\max}^{\delta}} \leq C_{\de}.
\end{align*}	

\end{proof}
Our next lemma treats all the terms on the right-hand side of \eqref{a:2} in 
 the form \\ $G(h,e^{-i t \p_x^2} f)=G(T(e^{-i t \p_x^2} f, e^{-i t \p_x^2} f), e^{-i t \p_x^2} f)$. 
Unfortunately, we cannot control such terms only 
with the {\it a posteriori} information $h\in \ch, e^{-i t \p_x^2} f\in X^{0,\f{1}{2}+\de}$. 
Instead, we must treat the whole expression as a trilinear one, which then 
yields the desired control. 
\begin{lemma}
\label{le:7}
Let $z$ satisfy $(\p_t+i\p_x^2) z = G(T(f,g), v_{>0})$ with $z(0,x) \equiv 0$, where $f,g\in X^{0,\frac{1}{2}+\delta}$.  Then for $0<10\delta<1-\alpha-\beta$,
\[
\|z\|_{X_T^{\al-\frac{1}{2}, \frac{1}{2}+\delta}} \leq C_{T,\delta} \|f\|_{X^{0,\frac{1}{2}+\delta}} \|g\|_{X^{0,\frac{1}{2}+\delta}} \|v\|_{X^{0,\frac{1}{2}+\delta}}.
\]
\end{lemma}
{\bf Remark:} In addition, we have similar control for $z_1, z_2$, so that \\  
$(\p_t+i\p_x^2) z_1 = G(T(f,g)_{>0}, v_{>0})$, $(\p_t+i\p_x^2) z_2 = G(T(f,g)_{\leq 0}, v_{>0})$. Namely 
\begin{equation}
\label{a:20}
\|z_1\|_{X_T^{\al-\frac{1}{2}, \frac{1}{2}+\delta}} + 
\|z_2\|_{X_T^{\al-\frac{1}{2}, \frac{1}{2}+\delta}}\leq C_{T,\delta} \|f\|_{X^{0,\frac{1}{2}+\delta}} \|g\|_{X^{0,\frac{1}{2}+\delta}} \|v\|_{X^{0,\frac{1}{2}+\delta}}.
\end{equation}
\begin{proof}
We handle the term $z_1$ first, after which, we quickly indicate how to treat the general case for $z$ (and $z_2=z-z_1$ is subsequently controlled). 

Applying Proposition~\ref{xsb} and the duality relation~\eqref{dual2}, 
\begin{eqnarray*}
\|z_1\|_{X_T^{\al-\frac{1}{2}, \frac{1}{2}+\delta}} &\leq & C_{T,\delta} 
\|G(T(f,g)_{>0}, v_{>0})\|_{X^{\alpha-\frac{1}{2}, -\frac{1}{2}+\delta}}\notag\\
	&\leq & C_{T,\delta} \sum_{k,l>0}  2^{\alpha k +\alpha l }  
	\|T(f,g)_k v_l\|_{X^{\beta-\f{1}{2},-\frac{1}{2}+\delta}}\\
	&\leq & C_{T,\delta} \sum_{k,l>0} 2^{\alpha k +\alpha l }  
	\sup_{\|w\|_{X^{\f{1}{2}-\be,\frac{1}{2}-\delta}_{\tau=\xi^2}}=1} 
	\left|\int_{\rone \times \rone} T(f,g)_k v_l \overline{w} \, dt\, dx\right|.	 
\end{eqnarray*}

We split the last sum  into 
\[ \sum_{l\leq k-3} \cdot +  \sum_{k\leq l-3} \cdot + \sum_{ k\sim l} \cdot 
\]
where $k,l>0$.  We denote the corresponding terms 
by $I_1 + I_2 + I_3$.  On each summand, we will apply Lemma~\ref{estimates} to obtain the desired estimate.  We need to estimate the integral
\begin{equation}\label{int}
 2^{\alpha k +\alpha l} \left|\int_{\rone \times \rone} T(f,g)_k v_l \overline{w} \, dt\, dx\right|.
\end{equation}

For $I_1$, we have high-low interaction between $T(f,g)_k$ and $v_l$ and hence \\ 
$T(f,g)_k v_l=P_{\sim k}[T(f,g)_k v_l]$.  Hence it suffices to control 
$$
2^{2 \alpha k} \left|\int_{\rone \times \rone} T(f,g)_k v_l \overline{w}_{\sim k} \, dt\, dx\right|.
$$
Thus 
 \begin{align*}
\eqref{int} &\leq C 2^{ 2 \alpha k}  \sum_{k_1,k_2>0}   \left|\int_{\rone\times \rone} T(f_{k_1}, g_{k_2})_k v_l \overline{w}_{\sim k} \, dt\,dx \right|\\
	 &\leq C 2^{(\alpha+\beta) k} \sum_{k_1,k_2>0}   2^{(\alpha-1)(k_1+k_2)} \left|\int_{\rone\times \rone} (f_{k_1} g_{k_2})_k v_l \overline{w}_{\sim k} \, dt\,dx \right|\\
	&\leq C 2^{ (\alpha+\beta) k} \sum_{k_1,k_2>0}   2^{(\alpha-1)(k_1+k_2)}
	\|(f_{k_1} g_{k_2})_k\|_{L^2_{t,x}}  \|v_l \overline{w}_{\sim k} \|_{L^2_{t,x}}. 
\end{align*}

Here we need to consider two cases: first when $2^{k_1} \sim 2^{k_2} \sim 2^k$, and second when this does not take place.  In the first case, we obtain
\begin{align*}
\eqref{int} &\leq C 2^{ (3\alpha+\beta-2) k}
	\|(f_{\sim k} g_{\sim k})_k\|_{L^2_{t,x}}  \|v_l \overline{w}_{\sim k} \|_{L^2_{t,x}}\\
&\leq  C_{\al, T,\delta} 2^{ (3\alpha+\beta-2) k}\|f\|_{L^8_T L^4_x} 
	\|g\|_{L^8_T L^4_x}  \|\overline{v}_l w_{\sim k} \|_{L^2_{t,x}}\\
	&\leq C_{\al, \delta,T} 2^{(3\alpha + 2\beta -3 +5\de)k }  \|f\|_{X^{0,\frac{1}{2}+\delta}} 
	\|g\|_{X^{0,\frac{1}{2}+\delta}}  \|v\|_{X^{0,\frac{1}{2}+\delta}} \|w\|_{X^{\f{1}{2}-\beta,\frac{1}{2}-\delta}}.
\end{align*}
where we have used the H\"older's inequality, the fact that $q=8, r=4$ is a Strichartz pair 
(and hence $\|f\|_{L^8_T L^4_x} \leq C_\de ||f||_{X^{0,\frac{1}{2}+\delta}}$) and the almost half a derivative gain that we derive from the estimates \eqref{kkk1} and \eqref{kkk2}.   On the other hand, if  $2^{k_1} \sim 2^{k_2} \sim 2^k$ does not take place, then we can gain $\f{1}{2}$ derivative from $\|(f_{k_1} g_{k_2})_{k}\|_{L^2}$ via \eqref{gain1} or \eqref{gain2}.  Thus the estimate follows
 \begin{align*}
\eqref{int}	&\leq C 2^{ (\alpha+\beta) k} \sum_{k_1,k_2>0} 2^{(\al-1)(k_1+k_2)}
	\|(f_{k_1} g_{k_2})_k\|_{L^2_{t,x}}  \|v_l \overline{w}_{\sim k} \|_{L^2_{t,x}}\\
	&\leq  C_{\al, \delta,T} 2^{(\alpha + \beta -1 +6\de)k }  \|f\|_{X^{0,\frac{1}{2}+\delta}} 
	\|g\|_{X^{0,\frac{1}{2}+\delta}}  \|v\|_{X^{0,\frac{1}{2}+\delta}} \|w\|_{X^{0,\frac{1}{2}-\delta}}.
\end{align*}
Clearly in both cases, the sums in $k>0$ goes through since $0<10\de<1-\al - \be$ and hence the bound \eqref{a:20} for this portion of the sum.\\

For $I_2$, note that now $T(f,g)_k v_l=P_{\sim l}[T(f,g)_k v_l]$ and hence 

 \begin{align}
\eqref{int} &\leq C 2^{\alpha (k+l)}\sum_{k_1,k_2>0}   \left|\int_{\rone\times \rone} T(f_{k_1}, g_{k_2})_k v_l \overline{w}_{\sim l} \, dt\,dx \right| \notag\\
	 &\leq C 2^{ \beta k + \alpha l }\sum_{k_1,k_2>0}   2^{(\alpha-1)(k_1+k_2)} \left|\int_{\rone\times \rone} (f_{k_1} g_{k_2})_k v_l \overline{w}_{\sim l} \, dt\,dx \right|. \label{w5}
\end{align}

If $2^{k_1} \sim 2^{k_2} \sim 2^l$,
\begin{align*}
\eqref{w5} &\leq C  2^{ (3\alpha+\beta-2)l} \|(f_{\sim l} g_{\sim l})_k\|_{L^2_{t,x}} \|v_l \overline{w}_{\sim l}\|_{L^2_{t,x}} \\
	&\leq C_{T,\de} 2^{(3\alpha +2\beta -3+6\de)l}  \|f\|_{X^{0,\frac{1}{2}+\delta}} 
	\|g\|_{X^{0,\frac{1}{2}+\delta}}  \|v\|_{X^{0,\frac{1}{2}+\delta}} \|w\|_{X^{\f{1}{2}-\beta,\frac{1}{2}-\delta}}.
\end{align*}

If $2^{k_1}\sim 2^{k_2}\sim 2^l$ does not hold, say $|k_1 -l|>6$.  Then because of the restriction $P_k (f_{k_1} g_{k_2})$ and $k \leq l-3$, it is not possible that $|k_2 - l|< 5$.    So in this case,  we would like to estimate the last integral by something close to 
$\|f_{k_1}v_l\|_{L^2_{t, x}} \|g_{k_2}\overline{w}_{\sim l} \|_{L^2_{t,x}}$ which would give us a full derivative gain in $l$. However, we cannot quite do that, since the integral in \eqref{w5} is not a pointwise product, but rather 
 the operator $P_k$ acting on $f_{k_1} g_{k_2}$, which then is multiplied by 
 $v_l\overline{w}_l$. \\
 
 The following 
calculation however provides a substitute for this, namely by Plancherel's and triangle inequality 
\begin{eqnarray*}
\left|\int_{\rone\times \rone} (f_{k_1} g_{k_2})_k v_l \overline{w}_{\sim l} \, dt\,dx \right|
& \leq &  
 \int_{\rone\times \rone} |\widehat{f_{k_1}}*\widehat{g_{k_2}}(\xi)|
\vp(2^{-k} \xi)|[ \widehat{v_l}* \widehat{\overline{w}_{\sim l}}](-\xi)| d\xi \, dt   \\
&\leq & \left|\int_{\rone\times \rone} Q[f_{k_1}] Q[g_{k_2}] Q[v_l] Q[\overline{w}_{\sim l}] dx\, dt\right|,
\end{eqnarray*}
where $Q[h]:=\cf^{-1}[|\hat{h}|]$. Note that $Q[h]_k=Q[h_k]$ and 
$\|Q[h]\|_{X^{s,b}}=\|h\|_{X^{s,b}}$, 
by the definition of $\|\cdot\|_{X^{s,b}}$. In other words, we have managed to remove the Littlewood-Paley operator $P_k$ (and to reduce to an expression as an integral 
of  pointwise product of four functions), at the expense of introducing the operators $Q$, which do not really affect the $X^{s,b}$  norms of the entries. 
With that last reduction in mind, we continue our estimation of \eqref{w5}.  

 ,  we have by \eqref{gain1} and \eqref{kkkk3} respectively 
\begin{eqnarray*}
\eqref{w5} &\leq & C \sum_{k_1,k_2>0} 2^{\beta k+ \alpha l +(\alpha-1)(k_1+k_2)} 
\|Q[f_{k_1}] Q[v_l]\|_{L^2_{t,x}} \|Q[g_{k_2}] Q[\overline{w}_l]\|_{L^2_{t,x}}\\
	&\leq & C_{T,\delta}  2^{\left(\alpha+\beta -1+6\delta\right)l } 
	\|f\|_{X^{0,\frac{1}{2}+\delta}}  \|v\|_{X^{0,\frac{1}{2}+\delta}}  
	\|g\|_{X^{0,\f{1}{2}+\de}} \|w\|_{X^{0,\f{1}{2}-\de}}. 
\end{eqnarray*}
 This clearly sums in $l$, provided $0<10\de < 1-\al-\be$. \\

For $I_3$, we have 

\begin{align}
\eqref{int} &\leq  C \sum_{k_1,k_2>0} 2^{(\alpha+\beta) k +(\alpha-1)(k_1+k_2)} \left|\int_{\rone\times \rone} (f_{k_1} g_{k_2})_k (v_{\sim k} \overline{w})_{\sim k} \, dt\,dx \right|. \label{w6}
\end{align}
At this point, let us discuss the frequency localization for $\overline{w}$. Clearly,   
\begin{equation}
\label{w66}
(v_{\sim k} \overline{w})_{\sim k}=(v_{\sim k} \overline{w}_{<k+3})_{\sim k}= 
(v_{\sim k} \overline{w}_{\sim k})_{\sim k}+ (v_{\sim k} \overline{w}_{\ll k})_{\sim k}, 
\end{equation} 
In particular, $ \overline{w}$ may not be high frequency. \\

If $|k_1-k|\geq 3$ or $|k_2-k| \geq 3$, then by \eqref{gain1} or \eqref{gain2} (applied to 
$\|(f_{k_1} g_{k_2})_k\|_{L^2_{t,x}}$)  and either \eqref{kkkk1} or \eqref{kkkk2} (applied to $\|(v_{\sim k} \overline{w})_{\sim k}\|_{L^2}$) 
\begin{align*}
\eqref{w6} &\leq C \sum_{k_1,k_2>0} 2^{(\alpha +\beta)k +(\alpha-1)(k_1+k_2)} 
\|(f_{k_1} g_{k_2})_k\|_{L^2_{t,x}}
(\|(v_{\sim k} \overline{w}_{\sim k})_{\sim k}\|_{L^2}+ 
\|(v_{\sim k} \overline{w}_{\ll k})_{\sim k})\|_{L^2})\\
	&\leq C_{\delta}  2^{\left(\alpha+\beta-1+6 \delta\right)k } \|f\|_{X^{0,\frac{1}{2}+\delta}}  \|g\|_{X^{0,\frac{1}{2}+\delta}}  
	\|v\|_{X^{0,\frac{1}{2}+\delta}} \|w\|_{X^{0,\frac{1}{2}-\delta}}.
\end{align*}
Summing in $k$ yields a bound, provided $0<10\de <1-\al-be$. 

Otherwise, $2^{k_1} \sim 2^{k_2}\sim 2^k$ and sums in $k_1, k_2$ are over a finite set.  In this case, we first handle the $w_{\sim k}$ term, which is easier due to the gain of $\f{1}{2}-\be$ derivatives in $k$.
\begin{align*}
\eqref{w6} &\leq C  2^{(3\alpha +\beta -2)k} \left|\int_{\rone\times \rone} (f_{\sim k} g_{\sim k})_k (v_{\sim k} \overline{w}_{\sim k})_{\sim k} \, dt\,dx \right|\\
	&\leq C 2^{(3\al+2\beta -3+5\de)  k} \|f\|_{X^{0,\frac{1}{2}+\delta}}  
	\|g\|_{X^{0,\frac{1}{2}+\delta}} \|v\|_{X^{0,\frac{1}{2}+\delta}} \|w\|_{X^{\al-\f{1}{2},\frac{1}{2}-\delta}}.
\end{align*}

For the term with $w_{\ll k}$, we need a more refined analysis, which is possible thanks to the estimate~\eqref{plusminus}.  We can write  
\begin{equation}\label{plusm}
(f_{\sim k} g_{\sim k})_k = (f_{\sim k}^+ g_{\sim k}^-)_k +(f_{\sim k}^- g_{\sim k}^+)_k +(f_{\sim k}^+ g_{\sim k}^+)_k^+ + (f_{\sim k}^- g_{\sim k}^-)_k^-.
\end{equation}

For the first two terms, due to \eqref{plusminus} and 
either \eqref{kkkk1} or \eqref{kkkk2} (applied respectively to the two terms arising in \eqref{w66}), we obtain 
\begin{align*}
\eqref{w6} &\leq C  2^{(\alpha +\beta) k } 
\|(f_{\sim k}^+ g_{\sim k}^-)_k\|_{L^2_{t,x}} 
\|(v_{\sim k} \overline{w}_{\ll k})_{\sim k}\|_{L^2_{t,x}}\\
	&\leq  C_\delta  2^{\left(\alpha+\beta-1+6 \delta\right)k } \|f\|_{X^{0,\frac{1}{2}+\delta}}  \|g\|_{X^{0,\frac{1}{2}+\delta}}  \|v\|_{X^{0,\frac{1}{2}+\delta}} 
	\|w\|_{X^{0,\frac{1}{2}-\delta}}.
\end{align*}

To deal with the next two terms in \eqref{plusm}, note that the integral in \eqref{w6} is in the form $\int (f_k^+ g_k^+)^+ (v_k^- \overline{w}_{\ll k} )_k^-$ or $\int (f_k^- g_k^-)^- (v_k^+ \overline{w}_{\ll k} )_k^+$.  We  estimate  the first one, the second one being symmetrically equivalent to the first. We need to once again apply the bounds involving the operator $Q$. 
\begin{align*}
\eqref{w6} &\leq C 2^{(\alpha+\beta)k} \left|\int_{\rone\times \rone} (f_{\sim k}^+ g_{\sim k}^+)_k (v_{\sim k}^- \overline{w}_{\ll k})_{\sim k}^- \, dt\,dx \right|\\
&\leq C  2^{(\alpha +\beta)k}  \left|\int_{\rone\times \rone} Q[f_{\sim k}^+] Q[g_{\sim k}^+] Q[v_{\sim k}^-] Q[\overline{w}_{\ll k}] dt\,dx \right|\\ 
	&\leq C  2^{(\alpha+\beta) k} 
	\|Q[f_{\sim k}^+] Q[v_{\sim k}^-]\|_{L^2_{t,x}} \|Q[g_{\sim k}^+] 
	Q[\overline{w}_{\ll k}]\|_{L^2_{t,x}}\\
	&\leq  C_\delta  2^{\left(\alpha+\beta-1+6\delta\right)k } \|f\|_{X^{0,\frac{1}{2}+\delta}}  \|g\|_{X^{0,\frac{1}{2}+\delta}}  \|v\|_{X^{0,\frac{1}{2}+\delta}} 
	\|w\|_{X^{0,\frac{1}{2}-\delta}}
\end{align*}
where we have applied \eqref{gain1} and  estimate \eqref{plusminus} for $\|Q[f_{\sim k}^+] Q[v_{\sim k}^-]\|_{L^2_{t,x}}$ and \eqref{kkkk2} for \\ $\|Q[g_{\sim k}^+] 
	Q[\overline{w}_{\ll k}]\|_{L^2_{t,x}}$. \\

Now to prove the same estimate for $z_2$, it is clear that the sums $I_1$ and $I_3$ do not appear (or is finite), since $k<0$ and $l>0$.  So we need to regard the sum of type $I_2$.  But here, instead of partitioning $(0,1)$ into sets of form $(2^{k-1}, 2^k)$, we will use a single characteristic function $\chi_{[0,1]}$ to represent the case when $k<0$ (thus we do not need to worry whether it is summable in $k$).  With the restriction $P_k (f_{k_1} g_{k_2})$, we must have $2^{k_1} \sim 2^{k_2}$.

 \begin{align}
\eqref{int} &\leq C \sum_{k_1>0} 2^{\alpha (k+l)}  \left|\int_{\rone\times \rone} T(f_{k_1}, g_{\sim k_1})_k v_l \overline{w}_{\sim l} \, dt\,dx \right| \notag\\
	 &\leq C \sum_{k_1 >0} 2^{\alpha l   +(2\alpha-2)k_1} \left|\int_{\rone\times \rone} (f_{k_1} g_{\sim k_1})_k v_l \overline{w}_{\sim l} \, dt\,dx \right|.\label{w7}
\end{align}

If $2^{k_1}\sim 2^l$,
\begin{align*}
\eqref{w7} &\leq C  2^{ \alpha l} \|(f_{\sim l} g_{\sim l})_k\|_{L^2_{t,x}} \|v_l \overline{w}_{\sim l}\|_{L^2_{t,x}} \\
	&\leq C_{T,\de} 2^{(\alpha  +\beta - 1+6\de)l}  \|f\|_{X^{0,\frac{1}{2}+\delta}} 
	\|g\|_{X^{0,\frac{1}{2}+\delta}}  \|v\|_{X^{0,\frac{1}{2}+\delta}} \|w\|_{X^{\f{1}{2}-\be,\frac{1}{2}-\delta}}.
\end{align*}

If $|k_1-l|>3$,

\begin{eqnarray*}
\eqref{w5} &\leq & C \sum_{k_1 >0} 2^{\alpha l +(2\alpha-2)(k_1)} 
\|Q[f_{k_1}] Q[v_l]\|_{L^2_{t,x}} \|Q[g_{\sim k_1}] Q[\overline{w}_l]\|_{L^2_{t,x}}\\
	&\leq & C_{T,\delta}  2^{\left(\alpha+\beta-1+6\delta\right)l } 
	\|f\|_{X^{0,\frac{1}{2}+\delta}}  \|v\|_{X^{0,\frac{1}{2}+\delta}}  
	\|g\|_{X^{0,\f{1}{2}+\de}} \|w_l\|_{X^{\f{1}{2}-\be,\f{1}{2}-\de}}. 
\end{eqnarray*}
\end{proof}

\subsection{Conclusion of the proof of Theorem \ref{theo:1}} 
\label{sec:s34}
Now that we have the needed multilinear estimates, we perform a 
fixed point argument for the solution  $w$ of \eqref{a:2}  in the space $\cx$. For simplicity, we group the terms on the right-hand side of \eqref{a:2} as follows\footnote{Recall that $G$ is a bilinear form}
\begin{eqnarray*}
  \cn_1 &=&  G([e^{-i t \p_x^2} f+w]_{\leq 0}, (Id+P_{> 0})[e^{-i t \p_x^2} f+h+w]); \\
 \cn_2 &=&  G(h_{\leq 0}, (Id+P_{> 0})[h+w])= \\ 
& =&   G(h_{\leq 0}, (Id+P_{> 0})[h])+
G(T(e^{-i t \p_x^2}f,e^{-i t \p_x^2}f)_{\leq 0}, (Id+P_{> 0})[w]);\\
 \cn_3 &=&  G(h_{\leq 0}, e^{-i t \p_x^2} (Id+P_{> 0})f)= 
 G(T(e^{-i t \p_x^2}f,e^{-i t \p_x^2}f)_{\leq 0},e^{-i t \p_x^2} (Id+P_{> 0})f);\\
\cn_4 &=& 2 G(e^{-i t \p_x^2} f_{>0}, h_{>0})=2 G(T(e^{-i t \p_x^2}f,e^{-i t \p_x^2}f)_{>0}, e^{-i t \p_x^2} f_{>0});  \\ 
 \cn_5  & = & 2 G(e^{-i t \p_x^2} f_{>0}, w_{>0});   \\
\cn_6 &=&  G(h_{>0}, h_{>0});\  \cn_7=G(h_{>0}, w_{>0})=G(T(e^{-i t \p_x^2}f,e^{-i t \p_x^2}f)_{>0}, w_{>0});\\
\cn_8 &=& G(w_{>0}, w_{>0}). 
\end{eqnarray*}
In order to finsh the proof, we need to show that  $w^0:=e^{-i t \p_x^2}[-T(f,f)]$ and \\
$\mw^j: (\p_t+i \p_x^2)\mw^j=\cn_j, \mw^j(0)=0,  j=1, \ldots, 8$, we have 
\begin{equation}
\label{a:101} 
\sum_{j=0}^8 \|\mw^j\|_{\cx}\leq C_{T, \de} (\|f\|_{L^2}+\|h\|_{\ch}+\|w\|_{\cx})^2(1+\|f\|_{L^2}+\|h\|_{\ch}+\|w\|_{\cx}). 
\end{equation}

\subsubsection{Estimates for $\mw^0$}
We have by Proposition \ref{xsb} and Lemma \ref{tmap1} (with $s=\al-\f{1}{2}<\f{1}{2}$) 
\begin{eqnarray*}
\|\mw^0\|_\cx= \|e^{-i t \p_x^2}[T(f,f)]\|_{X^{\al-\f{1}{2}, \f{1}{2}+\de}_T}\lesssim 
\|T(f,f)\|_{H^{\al-\f{1}{2}}}\lesssim \|f\|_{L^2}^2. 
\end{eqnarray*}
\subsubsection{Estimates for $\mw^1$}
Let us note first that for any two functions $\mu, \nu$,  \\
$
G(\mu_{\leq 0}, \nu)=
\langle \nabla \rangle^{\be-\al}(\langle \nabla \rangle^{\al} \mu_{\leq 0} \langle \nabla \rangle^{\al} \nu)
$ behaves for all practical purposes like $\mu_{\leq 0} \langle \nabla \rangle^{\be} \nu$. Thus, 
\begin{eqnarray*}
\|\mw^1\|_\cx &\leq & C_{T}  
\|G([e^{-i t \p_x^2} f+w]_{\leq 0},(Id+P_{> 0})[e^{-i t \p_x^2} f+h+w])\|_{L^2_t 
H^{\al-\f{1}{2}}_x}\\
&\leq &  C_{T}\|[e^{-i t \p_x^2} f+w]_{\leq 0}\cdot \langle \nabla \rangle^{\be} (Id+P_{> 0})[e^{-i t \p_x^2} f+h+w]\|_{L^2_t 
H^{\al-\f{1}{2}}_x}.
\end{eqnarray*}
By H\"older's inequality, we have 
$$
\|[e^{-i t \p_x^2} f+w]_{\leq 0}\cdot \langle \nabla \rangle^{\be} (Id+P_{> 0})[h]\|_{L^2_t 
H^{\al-\f{1}{2}}_x}\leq \|[e^{-i t \p_x^2} f+w]_{\leq 0}\|_{L^\infty_{t, x}} \|h\|_{L^2_t 
H^{\al+\be-\f{1}{2}}_x}.
$$
By the definition of $\ch$ however, $\ch\hookrightarrow X^{1-\de,\de} \hookrightarrow L^2_t H^{\al+\be-\f{1}{2}}_x$. Thus $\|h\|_{L^2_t 
H^{\al+\be-\f{1}{2}}_x}\leq C \|h\|_{\ch}$ and by Sobolev embedding and \eqref{a:1}
$$
\|[e^{-i t \p_x^2} f+w]_{\leq 0}\|_{L^\infty_{t, x}}\leq C 
\|e^{-i t \p_x^2} f+w\|_{L^\infty_{t} L^2_x}\leq C(\|f\|_{L^2}+\|w\|_{\cx}) 
$$
Regarding the remaining term in $\cn_1$, we can again split in two terms \\ 
$(Id+P_{> 0})[e^{-i t \p_x^2} f+w]=[e^{-i t \p_x^2} f+w]_{\leq 0}+
2 [e^{-i t \p_x^2} f+w]_{>0}$. The low frequency term is easy to deal with (by the argument above for $h$), whereas for the high-frequency term, we have by Lemma \ref{estimates} 
(more specifically \eqref{gain1}) that 
\begin{align*}
& \|[e^{-i t \p_x^2} f+w]_{\leq 0}\cdot \langle \nabla \rangle^{\be} [e^{-i t \p_x^2} f+w]_{> 0}\|_{L^2_t 
H^{\al-\f{1}{2}}_x}\\
&\leq  \|[e^{-i t \p_x^2} f+w]_{\leq 0}\|_{X^{0,\f{1}{2}+\de}}  
\sum_{k>0} 2^{-(\f{1}{2}-\de)k} 2^{\al+\be-\f{1}{2}k} \|[e^{-i t \p_x^2} f+w]_{k}\|_{X^{0,\f{1}{2}+\de}} \\
&\leq  C (\|f\|_{L^2}+\|w\|_{X^{0,\f{1}{2}+\de}})^2. 
\end{align*}

\subsubsection{Estimates for $\mw^2$}
Write $\mw^2=\mw^{2}_1+\mw^2_2$, where $\mw^2_1$ is the solution corresponding from the first term in $\cn_2$. Then,  
\begin{align*}
\|\mw^{2}_1\|_{\cx}&\leq C_T \|G(h_{\leq 0}, (Id+P_{> 0})[h])\|_{L^2_t H^{\al-\f{1}{2}}_x}\leq C \|h_{\leq 0}\|_{L^\infty_{t, x}} \|h\|_{L^2_t H^{\al-\f{1}{2}}_x}\\
&\leq C \|h_{\leq 0}\|_{L^{\infty}_t H^{\frac{1}{2}}_x} \|h\|_{\ch}\leq C \|h\|_{\ch}^2. 
\end{align*}
since $\ch\hookrightarrow L^2_t H^{\al-\f{1}{2}}_x \cap L^{\infty}_t H^{\frac{1}{2}}_x$. 

As far as $\mw^2_2$ is concerned, we apply Lemma \ref{le:7} (more precisely \eqref{a:20} for $z_2$), which yields 
$$
\|\mw^{2}_2\|_{\cx}\leq C_T \|e^{i t \p_x^2} f\|_{X^{0,\f{1}{2}+\de}}^2 
\|(Id+P_{>0})w\|_{X^{0,\f{1}{2}+\de}}\leq C \|f\|_{L^2}^2 \|w\|_\cx. 
$$
\subsubsection{Estimates for $\mw^3$}
The estimate for $\mw^3$ is pretty similar to the one for $\mw^2_2$. By \eqref{a:20}, applied for $z_2$,  
$$
\|\mw^3\|_{\cx}\leq C_T \|e^{-i t \p_x^2} f\|_{X^{0,\f{1}{2}+\de}}^2 
\|e^{-i t \p_x^2}[(Id+P_{>0})f]\|_{X^{0,\f{1}{2}+\de}}\leq C \|f\|_{L^2}^3. 
$$

\subsubsection{Estimates for $\mw^4$}
We have by \eqref{a:20}, applied for $z_1$,  that 
$$
\|\mw^4\|_\cx\leq C \|e^{-i t \p_x^2} f_{>0}\|_{X^{0,\f{1}{2}+\de}}^3 \leq C\|f\|_{L^2}^3. 
$$
\subsubsection{Estimates for $\mw^5$}

For $\mw^5$, we apply Lemma \ref{le:5}, whence 
$$
\|\mw^5\|_\cx\leq C \|w\|_{X^{\al-\f{1}{2},\f{1}{2}+\de}}
\|e^{-i t \p_x^2} f\|_{X^{0,\f{1}{2}+\de}}\leq C \|w\|_{\cx}\|f\|_{L^2}. 
$$
\subsubsection{Estimates for $\mw^6$}
The estimate for $\mw^6$ follows form Lemma \ref{le:6}, 
$$
\|\mw^6\|_\cx\leq C \|h\|_{X^{1-\de, \de}}^2\leq C \|h\|^2_\ch. 
$$
\subsubsection{Estimates for $\mw^7$}
This terms is in fact simpler than $\mw^4$ (since $w$ in the second component 
is in fact smoother than the free solution in $\cn_4$).  We deal with it in the same way. Namely, by \eqref{a:20}, applied to $z_1$, we have 
$$
\|\mw^7\|_\cx\leq C\|e^{-i t \p_x^2} f\|_{X^{0,\f{1}{2}+\de}}^2 \|w\|_{X^{0,\f{1}{2}+\de}}\leq 
C \|f\|_{L^2}^2 \|w\|_\cx. 
$$
 \subsubsection{Estimates for $\mw^8$}
Finally, the estimate for $w^8$ follows from Lemma \ref{le:5}. We have 
$$
\|\mw^8\|_\cx\leq C\|w\|_{X^{\al-\f{1}{2}, \f{1}{2}+\de}} \|w\|_{X^{0, \f{1}{2}+\de}}\leq 
C \|w\|_{\cx}^2
$$

 \section{Regarding l.w.p. for  nonlinearities of the form  $\nab^{\beta}[u \bar{u}]$ and $\nab^{\beta}[\bar{u}^2]$}
 \label{last}
 We will just briefly sketch the analysis that one needs to undertake, in order to pursue well-posedness of the problem 
 $$
 u_t+i u_{xx}= \nab^{\beta}[u \bar{u}].
 $$
As a byproduct of this discussion, we will hopefully be able to shed some light on the issue with low regularity, which is present in this particular case, \cite{LW}. 

Following the ideas of Section \ref{sec:s3}, we need to construct $T$, so that \eqref{8} is satisfied, where of course $G(u,v)=\nab^{\be-\al}[\nab^{\al} v \nab^{\al} \bar{v}]$. It is easy to see that the needed $T$ is in the form  
\begin{equation}
\label{s:40}
T(u,v)(x)= \f{1}{8\pi^2 i} \int 
\f{\langle\xi\rangle^\al\langle\eta\rangle^\al}{\langle\xi+\eta\rangle^{\al-\be}}
\f{1}{ \xi (\xi+\eta)} \widehat{u_{>0}}(\xi) \widehat{\bar{v}}(\eta) 
e^{i(\xi+\eta)x} d\xi 
d\eta.
\end{equation} 
Note that this transformation may be performed only when the output function $T(u,v)$ is Fourier localized, so that its frequency satisfies  
 $\gtrsim 1$, so that we do not run into trouble with the term $(\xi+\eta)^{-1}$ inside the symbol of $T$.  This is the reason why, in general (and unless we impose some homogeneous Sobolev norms in small frequencies, as is done in \cite{LW}),  we cannot do better than $H^{-\f{1}{4}+}$ l.w.p. 
 
 It is also clear from the form \eqref{s:40}, that in the case of ``high-high'' interactions, the (generally smoothing) term $(\xi+\eta)^{-1}$ is not of much help to achieve better smoothness of $T(u,v)$. Therefore, performing this transformation would be advantageous, only if $2\al<1$. This is a simple (if a little naive)  way to see the optimality of  restriction $\al<1/2$ in the results of \cite{LW}.  

For the nonlinearities of the form $\nab^{\beta}[\bar{u}^2]$, following the same ideas, we come up with the following normal form 
\begin{equation}
\label{s:50}
T(u,v)(x)= \f{1}{8\pi^2 i} \int 
\f{\langle\xi\rangle^\al\langle\eta\rangle^\al}{\langle\xi+\eta\rangle^{\al-\be}}
\f{1}{(\xi^2+\eta^2+\xi \eta)} \widehat{u}(\xi) \widehat{\bar{v}}(\eta) 
e^{i(\xi+\eta)x} d\xi 
d\eta.
\end{equation} 
Clearly, this normal form gains a derivative in each variable (very similar to the case $\nab^{\be}[u^2]$) and hence, one may expect to get an identical result to Theorem \ref{theo:1} for this nonlinearity as well. We will not pursue more details.

\end{document}